\documentclass[11pt,a4paper,abstract=on]{scrartcl}

\usepackage[english]{babel}

\usepackage{lmodern}
\usepackage[T1]{fontenc}

\usepackage{ucs}
\usepackage[utf8x]{inputenc}

\usepackage{amsfonts,amstext,amsmath,amssymb,amsopn,amsthm}

\usepackage{mathtools}

\usepackage{dsfont}
\usepackage{mathrsfs}

\usepackage{listings}
\newcommand{\dd}{\mbox{d}}
\usepackage[dvipsnames, svgnames, x11names]{xcolor}

\usepackage{siunitx}
\usepackage{graphicx}
\usepackage{subfig}
\usepackage{wrapfig}

 \usepackage{todonotes}

\usepackage{csquotes}
\usepackage[colorlinks=true, 
linkcolor=blue,
pdfstartview=FitH,      
breaklinks=true,        
bookmarksopen=true,     
bookmarksnumbered=true  
]{hyperref}
\usepackage{enumerate}
\usepackage{caption}

\usepackage[top=1in, bottom=1.5in, left=1in, right=1in]{geometry}
\usepackage[linesnumbered,lined,ruled, noend]{algorithm2e}
\usepackage{harvard}

\newcommand{\paren}[1]{\left(#1\right)}
\newcommand{\ecklam}[1]{\left[#1\right]}

\renewcommand{\equiv}{\ensuremath{:=}}

\newcommand{\mymap}[3]{#1:\,#2 \to #3\,}

\newcommand{\mysetc}[2]{\left\{#1\,\middle|\,#2\right\}}
\newcommand{\set}[2]{\left\{#1\,\middle|\,#2\right\}}

\newcommand{\ip}[2]{\left\langle #1,\, #2\right\rangle}

\newcommand{\cpr}[2]{\ensuremath{\mathbb{P}\left(#1\,\middle|\,#2\right)}}

\newcommand{\norm}[1]{\left\|#1\right\|}

\newcommand{\Nbb}{\mathbb{N}}

\newcommand{\Rn}{\mathbb{R}^n}
\newcommand{\Pcal}{\mathcal{P}}
\newcommand{\Jcal}{\mathcal{J}}
\newcommand{\Fcal}{\mathcal{F}}

\newcommand{\xbar}{{\overline{x}}}

\newcommand{\tbar}{\overline{t}}

\DeclareMathOperator{\dist}{dist}

\DeclareMathOperator*{\argmin}{\arg\!\min}

\DeclareMathOperator{\prox}{prox}

\DeclareMathOperator{\Id}{Id}

\DeclareMathOperator{\Fix}{Fix}

\DeclareMathOperator{\gph}{gph}

\DeclareMathOperator{\inv}{inv}
\DeclareMathOperator{\indep}{\perp \!\!\! \perp\,}

\makeatletter
\def\@endtheorem{\endtrivlist\@endpefalse }
\makeatother

\newtheorem{thm}{Theorem}[section]
\newtheorem{cor}[thm]{Corollary}
\newtheorem{lemma}[thm]{Lemma}
\newtheorem{prop}[thm]{Proposition}
\theoremstyle{definition}
\newtheorem{example}[thm]{Example}
\newtheorem{definition}[thm]{Definition}
\newtheorem{assumption}[thm]{Assumption}

\newtheoremstyle{note}
{3pt}
{3pt}
{}
{}
{\bfseries}
{\bfseries :}
{.5em}
{}
\theoremstyle{note}

\newtheorem{rem}[thm]{Remark}

\title{Rates of Convergence for Chains of Expansive Markov Operators}
\author{Neal Hermer\thanks{Institute for Numerical and Applied Mathematics,
    University of Goettingen,
    37083 Goettingen, Germany. NH was supported by 
    Deutsche Forschungsgemeinschaft Research Training Grant 2088 TP-B5.
    E-mail:  \texttt{n.hermer@math.uni-goettingen.de}}, 
  D. Russell Luke\thanks{Institute for Numerical and Applied Mathematics,
    University of Goettingen,
    37083 Goettingen, Germany. DRL was supported in part by 
    Deutsche Forschungsgemeinschaft Research Training Grant 2088 TP-B5 
    and 
	Deutsche Forschungsgemeinschaft (DFG, German Research Foundation) - Project-ID 432680300 - SFB 1456.
    E-mail:  \texttt{r.luke@math.uni-goettingen.de}}  
  and  Anja Sturm\thanks{Institute for Mathematical Stochastic,
    University of Goettingen,
    37077 Goettingen, Germany. AS was supported in part by Deutsche 
Forschungsgemeinschaft 
    Research Training Grant 2088 TP-B5 and 
	Deutsche Forschungsgemeinschaft (DFG, German Research Foundation) - Project-ID 432680300 - SFB 1456.
    E-mail:  \texttt{asturm@math.uni-goettingen.de}}}
\date{\today}

\begin{document}
 \maketitle

 \begin{abstract}
We provide conditions that guarantee local rates of convergence in distribution of iterated random functions
   that are  {\em not nonexpansive} mappings in locally compact Hadamard spaces.  
   Our results are applied to stochastic instances of common algorithms in optimization, 
   stochastic tomography for X-FEL imaging, and a stochastic algorithm for the computation 
   of Fr\'echet means in model spaces for phylogenetic trees. 
\end{abstract}

{\small \noindent {\bfseries 2010 Mathematics Subject Classification:}
  Primary 60J05, 
  46N10, 
  46N30, 
  65C40, 
  49J55 
    Secondary  49J53,   
    65K05.\\ 
}

\noindent {\bfseries Keywords:}
Averaged mapping, nonexpansive mapping, stochastic feasibility,
inconsistent stochastic fixed point problem, iterated random
functions, convergence, Markov chain, Markov operator

\section{Introduction}
\label{sec:introduction}
This work concerns abstract mathematical structures that guarantee convergence in distribution of 
Markov chains when the transition kernel is generated by mappings that are not {\em nonexpansive}.
Our motivation for the abstract study pursued here is 
the very concrete application of X-ray free electron laser (X-FEL) imaging experiments \cite{Chap11,Bou12,ArdGru}.
Our framework provides conditions under which algorithms for the reconstruction of {\em electronic densities} or 
{\em orbitals} of a molecule from X-FEL 
measurements are guaranteed to be successful.
Another motivating application comes from already established stochastic algorithms for computing 
the Fr\'echet means of phylogenetic trees in nonlinear model spaces.  For this problem our analysis
immediately yields improvements to the state of the art for these kinds of algorithms.  
While our focus is mainly mathematical, we briefly introduce the X-FEL imaging problem to set the scene. 

The goal of X-FEL imaging is to determine the {\em electron density} of a molecule from 
experimental samples of its scattering probability distribution (i.e. diffraction pattern). Without going into the 
intricate aspects of the measurement device or scattering theory, we imagine a three-dimensional
measurement device  that counts the occurrence of photons in a three-dimensional domain partitioned 
into voxels of a fixed size far away from (that is, in the {\em far field} of) a molecule  
that has been illuminated by a short X-FEL pulse.  The 
molecule under observation is at a random orientation relative to the 
measurement device.  The illuminating pulse is only long enough to cause a few (between 3 and 100) 
scattering events from the interaction of a molecule's electrons with the X-ray.  The 
experiment is repeated about $10^9$ times, each time with the molecule at a different random orientation. 

The physical 
model mapping the probability of a scattering event at a particular location on the molecule to the 
resulting electromagnetic field at large distances, is, to first order, simply the Fourier transform
of the {\em electronic density} of the molecule.  An X-FEL measurement is truly a random sample of the 
electromagnetic field far away from the molecule.  
The measurement device, however, can only measure the occurrence of photons, not their direction as 
indicated by the real and complex parts of the electromagnetic field.  If both the probability of observing 
a photon at a certain location, and its direction -- i.e. its phase -- were known, then the electronic 
density of the molecule that 
produced those photons could be reconstructed from the observation by simply inverting the Fourier transform.   
When the probabilities of observing 
photons at each location in the far field of the molecule can be measured (as is the case in 
conventional coherent diffraction imaging \cite{RobSal20} and wavefront sensing \cite{Luke02a}), 
then the problem becomes 
the well-known {\em phase retrieval problem}  \cite{Rayleigh92} for 
which the mathematical analysis has been developed 
in \cite{Luke02a}, \cite{BurkeLuke03}, \cite{LukNguTam18}, and \cite{LukMar20}.

The algorithms that we envision for performing such reconstructions from X-FEL data proceed by 
constructing a mapping from the current estimate for the electron density of the molecule 
to another electron density that is, in some sense, closer to being consistent with new 
X-FEL observations;  in other words, the mapping is itself stochastic. Examples of such mappings
are gradient flows of the parameterization of the the density in a direction that either maximizes
the probability of the new observation or minimizes some energy;  alternatively, one could envision 
{\em projections} of the current estimate onto sets of ``most likely'' densities that explain the 
new observations.  We view this type of algorithmic procedure as a random function iteration generating 
a Markov chain whose {\em transition kernel} maps the current estimate to a new estimate based on new randomly sampled data.  

In \cite{HerLukStu22a} it is shown that iterations of randomly selected 
{\em $\alpha$-firmly nonexpansive mappings} (defined precisely below) in a Euclidean space converge in the 
Prokhorov-L\'evy metric to an invariant probability measure of the corresponding Markov operator.  
Mappings that are $\alpha$-firmly nonexpansive are also {\em nonexpansive}; 
the analysis of Markov chains for the case of nonexpansive 
transition kernels in queuing theory can be found in \cite{BaccBrem03} and \cite{BaccMair98}.  
The transition kernels that we propose for the X-FEL application, however, are not nonexpansive. 
Our analysis extends the more recent study 
of nonexpansive transition kernels in \cite{HerLukStu22a} to the expansive case
using results from the deterministic theory of expansive fixed point mappings developed in \cite{LukNguTam18}.
We return to the specific application of X-FEL imaging in Section \ref{sec:incFeas}.  

Random function iterations (RFI) \cite{Diaconis1999} generalize
deterministic fixed point iterations, and are a useful framework for
studying a number of important applications.  An RFI is a stochastic
process of the form $ X_{k+1} := T_{\xi_{k}} X_{k} $ ($k=0,1,2,\dots$)
initialized by a random variable $X_0$ with distribution $\mu_0$ and
values on some set $G$.  This of course includes initialization from
a deterministic point $x_0\in G$ via the $\delta$-distribution.  
Here $\xi_{k}$ ($k=0,1,2,\dots$) is an
element of a sequence of i.i.d.\ random variables that map from a
probability space into a measurable space of indices $I$ (not
necessarily countable) and $T_{i}$ $(i \in I)$ are self-mappings on
$G$.  The iterates $X_{k}$ form a (time-homogeneous) Markov chain that takes values in
the space $G$, which is, for our purposes, a \emph{Polish space}.  
Deterministic fixed point iterations are included when the index set
$I$ is just a singleton.

Notation, basic facts, and the main results are presented in Section \ref{sec:consistent_feas_prob}.
The assumptions on the mappings generating the Markov operators have been shown to be 
necessary and sufficient for local rates of convergence in 
the analysis of deterministic algorithms for continuous optimization \cite[Theorem 2]{LukTebTha18}
and \cite[Theorem 16]{LauLuk21}, and also necessary for rates for the case of 
consistent stochastic feasibility \cite[Theorem 3.15]{HerLukStu19a} in a Euclidean setting. 
Whether these assumptions are also necessary in the setting considered here is not known, though we
conjecture that this is true.  
In Section \ref{sec:main statements}  we present the statement of the main result,   
Theorem \ref{t:msr convergence} 
which provides sufficient conditions for a quantification of convergence of the RFI. 
The proof of the main result is delayed until Section \ref{sec:mainres} so that the  
regularity assumptions can be developed;  namely, that the transition kernel is {\em almost}
$\alpha$-firmly nonexpansive {\em in expectation} (Definition \ref{d:afne oa}) 
and the {\em discrepancy} between a given measure and the set of invariant 
measures of the Markov operator, \eqref{eq:Psi},  
is {\em metrically subregular} (Definition \ref{d:(str)metric (sub)reg}). 
We conclude this study with Section \ref{sec:incFeas} where 
we focus on applications to nonconvex optimization on measure spaces and 
inconsistent feasibility.  

\section{RFI and the Stochastic Fixed Point Problem}
\label{sec:consistent_feas_prob}

This section follows the development of \cite{HerLukStu22a}.  
Our notation is standard.  The natural numbers are denoted by $\mathbb{N}$.  
For $G$, an abstract topological space,  
$\mathcal{B}(G)$ denotes the Borel $\sigma$-algebra and 
$(G,\mathcal{B}(G))$ is the 
corresponding measure space.  We denote by $\mathscr{P}(G)$ the set of all
probability measures on $G$.  The notation $X \sim \mu\in \mathscr{P}(G)$ means that the law
 of $X$, denoted $\mathcal{L}(X)$, satisfies 
 $\mathcal{L}(X)\equiv\mathbb{P}(X \in \cdot) = \mu$, 
 where $\mathbb{P}$ is the probability
 measure on some underlying probability space.  All of these different ways of 
 indicating a measure $\mu$ will be used.  We take for granted a familiarity
 with probability theory and Markov chains.  For general background  and fundamental results, in particular regarding the convergence of 
 Markov chains, 
see for instance \cite{kallenberg1997},  \cite{Billingsley},  \cite{Hairer2006},
 \cite{Hairer2021},   \cite{HerLerLas},  \cite{MeyTwe},  \cite{LevinPeresWilmer2017},  
 \cite{Villani2008}, and \cite{Walters82}. 
%
We emphasize that the results of the standard probability texts on Markov chain convergence 
 are not directly applicable to our particular setting, which includes the special conditions that are placed on the operators $T_i$.

Throughout, the pair $(G,d)$ denotes a separable metric space with metric
$d$.  Results concerning existence
of invariant measures and convergence of Markov chains require {\em completeness} of the 
metric space, that is that the space is {\em Polish}.  In defining the regularity of 
the building blocks, however, completeness is 
not a part of the characterization. For our main result, Theorem \ref{t:msr convergence}, we restrict the setting 
to self-mappings of a compact subset of a Hadamard space, a separable complete uniformly 
convex metric space with nonpositive curvature.  
Our application examples (Section \ref{sec:incFeas}) are mostly in Euclidean spaces, which are Hadamard 
spaces with zero curvature, but the example of computing mean phylogenetic trees (Section \ref{sec:tree})  
requires the generality of Hadamard spaces.  The purpose of starting 
with abstract separable metric spaces is to make it easier to see which results rely in part  
on the regularity of the space in addition to those properties we require of the mappings.  
Markov operators generated from mappings on Hadamard spaces represent the current limits of 
the direct applicability of the theory presented here.  

The distance of a point $x\in G$ to a set 
$A\subset G$ is denoted by $d(x,A)\equiv \inf_{w\in A}d(x,w)$. 
Continuing with the development initiated in the introduction, we will consider 
a collection of mappings $\mymap{T_{i}}{G}{G}$, $i \in I$, on $(G,d)$, 
where $I$ is an arbitrary index set - not necessarily countable.    
The measure space of indexes is denoted by
$(I,\mathcal{I})$, and $\xi$ is an $I$-valued random variable on a
probability space. 
The pairwise independence of two random variables $\xi$ and $\eta$ is
denoted $\xi\indep\eta$.  The random variables $\xi_k$ in the sequence
$(\xi_{k})_{k\in\mathbb{N}}$ (abbreviated $(\xi_{k})$) 
are independent and identically
distributed (i.i.d.) \ with $\xi_{k}$ distributed as
$\xi$ ($\xi_k\sim \xi$).  
The method of random
function iterations is formally presented in Algorithm \ref{algo:RFI}.
\begin{algorithm}    
\SetKwInOut{Output}{Initialization}
  \Output{Set $X_{0} \sim \mu_0 \in \mathscr{P}(G)$, $X_{0} \indep (\xi_{k})$ with  $\xi_k\sim\xi$ i.i.d.}
    \For{$k=0,1,2,\ldots$}{
            {$ X_{k+1} = T_{\xi_{k}} X_{k}$}\\
    }
  \caption{Random Function Iterations (RFI)}\label{algo:RFI}
\end{algorithm}

\noindent We will use the notation
\begin{equation}\label{eq:X_RFI}
  X_{k}^{X_0} := T_{\xi_{k-1}} \ldots T_{\xi_{0}} X_{0}
\end{equation}
to denote the sequence of the RFI initialized with $X_0\sim \mu_0$.
When characterizing sequences initialized
with the delta distribution of a point $x$ we use the notation $X_{k}^{x}$.
The following assumptions will be employed
throughout.

\begin{assumption}\label{ass:1}
  \begin{enumerate}[(a)]
  \item\label{item:ass1:indep} $\xi_{0},\xi_{1}, \ldots, \xi_{k}$
    are i.i.d with values on $I$ and $\xi_k\sim \xi$.  $X_0$ is a 
    random variable with values on $G$, independent from $\xi_k$. 
  \item\label{item:ass1:Phi} The mapping $\mymap{\Phi}{G\times
      I}{G}$, $(x,i)\mapsto T_{i}x$ is measurable.
  \end{enumerate}
\end{assumption}

We define the Markov chains corresponding to the random function iteration in terms of its 
{\em transition kernel}:  $\mymap{p}{G\times \mathcal{B}(G)}{[0,1]}$ where  
$p(\cdot,A)$ is measurable for all $A \in
\mathcal{B}(G)$  and $p(x,\cdot)$ is a probability measure for all $x \in G$.
  A sequence of random variables $(X_{k})$,
  $\mymap{X_{k}}{(\Omega,\mathcal{F},\mathbb{P})}{(G,\mathcal{B}(G))}$
  is called a Markov chain with transition kernel $p$ if for all $k \in
  \mathbb{N}$ and $A \in \mathcal{B}(G)$
  $\mathbb{P}$-a.s.\ the following hold:
  \begin{enumerate}[(i)]
  \item $\cpr{X_{k+1} \in A}{X_{0}, X_{1}, \ldots, X_{k}} =
    \cpr{X_{k+1} \in A}{X_{k}}$;
  \item $\cpr{X_{k+1} \in A}{X_{k}} = p(X_{k},A)$.
  \end{enumerate}

  Under Assumption \ref{ass:1},
  the sequence of random variables $(X_{k})$
  generated by Algorithm \ref{algo:RFI} is a Markov chain with transition
  kernel $p$ given by 
\begin{equation}\label{eq:trans kernel}
  (x\in G) (A\in
  \mathcal{B}(G)) \qquad p(x,A) \equiv \mathbb{P}(\Phi(x,\xi) \in A) =
  \mathbb{P}(T_{\xi}x \in A)
\end{equation}
for the measurable \emph{update function} 
$\mymap{\Phi}{G\times  I}{G}$, $(x,i)\mapsto T_{i}x$ \cite[Proposition 2.3]{HerLukStu22a}.

The Markov operator $\mathcal{P}$ is defined pointwise for a measurable 
function 
$\mymap{f}{G}{\mathbb{R}}$ via
\begin{align*}
  (x\in G)\qquad \mathcal{P}f(x):= \int_{G} f(y) p(x,\dd{y}),
\end{align*}
when the integral exists. 

The Markov operator $\mathcal{P}$ 
is \emph{Feller} if $\mathcal{P}f \in C_{b}(G)$ whenever $f \in
C_{b}(G)$, where $C_{b}(G)$ is the set of bounded and continuous
functions from $G$ to $\mathbb{R}$. 
This property is central to the theory of existence of invariant measures. 
An application of Lebesgue's dominated convergence theorem immediately yields that
if $T_{i}$ is continuous for all $i\in I$
then the Markov operator $\mathcal{P}$ is Feller, see also  \cite[Theorem 4.22]{Hairer2006}. 

We denote the
dual Markov operator $\mymap{\mathcal{P}^{*}}{\mathscr{P}(G)}{\mathscr{P}(G)}$ 
acting on a measure $\mu$ by action on 
the right by $\mathcal{P}$ via
\begin{align*}
  (A \in \mathcal{B}(G))\qquad (\mathcal{P}^{*}\mu) (A):=
  (\mu\mathcal{P}) (A) := \int_{G} p(x,A) \mu(\dd{x}).
\end{align*}
We can thus write the distribution of the $k$-th 
iterate of the Markov chain generated
by Algorithm \ref{algo:RFI} as $\mathcal{L}(X_{k}) =
\mu_0 \mathcal{P}^{k}$.

\subsection{The Stochastic Fixed Point Problem}\label{sec:consist RFI}
The {\em stochastic feasibility}
problem is:
\begin{align}
  \label{eq:stoch_feas_probl}
\mbox{ Find }  x^{*} \in C := \mysetc{x \in G}{\mathbb{P}(x = T_{\xi}x) = 
1}.
\end{align}
This was first studied in \cite{ButnariuFlam95} and  \cite{Butnariu95} in the context of 
{\em stochastic convex set feasibility} and has been extended more generally to 
finding the (almost sure) fixed points of multi-valued mappings in 
\cite{HerLukStu19a} and   \cite{HerLukStu22a}.
A point $x$ such that $x = T_{i}x$ is a {\em fixed point} of the operator $T_{i}$;
the set of all such points is denoted by 
\begin{align*}
  \Fix T_{i} = \mysetc{x \in G}{x = T_{i}x}.
\end{align*}
In \cite{HerLukStu19a} it was assumed that $C\neq\emptyset$.  If
$C=\emptyset$, following Butnariu \cite{Butnariu95} this is called the {\em inconsistent stochastic
  feasibility}.  
  
One need only consider linear systems of equations, $Ax=b$, to demonstrate an inconsistent feasibility
problem.  Each linear equation, $\ip{a_j}{x}=b_j$ represents a hyperplane;  the solution to $Ax=b$
is the intersection of the hyperplanes.  Obviously, when the system is overdetermined, or 
underdetermined with noise, the intersection is likely empty - an inconsistent feasibility
problem.  Most readers will reflexively reformulate the problem as a least squares minimization 
problem, but there are compelling reasons why this is not always the best way around the dilemma 
of the inconsistent feasibility problem.  

Avoiding debate about least-squares versus feasibility, 
what all can surely agree upon is that the inconsistency of the problem formulation is an artifact 
of asking the wrong question.  A fixed point of the dual Markov operator $\mathcal{P}$
is called an \emph{invariant} measure and satisfies $\pi \mathcal{P} = \pi$.  
The set of all invariant probability measures is denoted by 
$\inv \mathcal{P}$.
The solution to the dilemma of the inconsistent feasibility problem is instead to 
solve the following {\em stochastic fixed point problem}:
\begin{align}
  \label{eq:stoch_fix_probl}
  \mbox{Find}\qquad \pi\in\inv\mathcal{P}.
\end{align}

\begin{example}[inconsistent stochastic feasibility]\label{example:inconsist}
  Consider the (trivially convex, nonempty and closed) sets
  $C_{-1}\equiv \{-1\}$ and $C_1\equiv \{1\}$ together with a random
  variable $\xi$ such that $\mathbb{P}(\xi=1)=\mathbb{P}(\xi=-1)=1/2$.
  The mappings $T_i x=P_{C_i}x=i$ for $x \in \mathbb{R}$ and $i\in
  I=\{-1,1\}$ are the projections onto the sets $C_{-1}$ and $C_{1}$.
  The RFI iteration then amounts to just random jumps between the
  values $-1$ and $1$. So it holds that $\mathbb{P}(T_{\xi}x = i) =
  1/2$ for all $x \in \mathbb{R}$ and hence there is clearly no
  feasible fixed point to this iteration, that is, the set
  $C$ defined in \eqref{eq:stoch_feas_probl} is empty. Nevertheless, by {\em disintegration} we have
  \begin{align*}
    \mathbb{P}(X_{k+1} = i) = \mathbb{E}[\cpr{T_{\xi_{k}}X_{k} =
      i}{X_{k}}] = \frac{1}{2}
  \end{align*}
  for all $k \in \mathbb{N}$. That means the unique invariant
  distribution to which the distributions of the iterates of the RFI
  (i.e.\ $(\mathbb{P}(X_{k} \in \cdot))_{k}$) converges is
  $\pi=\tfrac{1}{2}(\delta_{-1}+\delta_{1})$, and this is attained
  after one iteration.  The least squares solution is the mean of this
  invariant distribution.
\end{example}

\subsection{Modes of convergence}
\label{sec:modesofcvg}
Algorithm \ref{algo:RFI} generates a sequence of random variables whose distributions, 
when conditions allow, converge to solutions to \eqref{eq:stoch_fix_probl}.
For this we focus on {\em convergence in distribution}.
Let $(\mu_{k})$ be a sequence of probability measures on $G$.  The sequence 
$(\mu_{k})$ is said to converge in distribution to $\mu$ whenever $\mu \in
\mathscr{P}(G)$ and for all $f \in C_{b}(G)$ it holds that $\mu_{k} f
\to \mu f$ as $k \to \infty$, where $\mu f := \int f(x) \mu(\dd{x})$. 
Equivalently a sequence of random variables $(X_{k})$ is said to converge in 
distribution if their laws $(\mathcal{L}(X_{k}))$ do.

We consider convergence in distribution for the corresponding sequence of measures
  $(\mathcal{L}(X_{k}))$ to a probability measure $\pi \in \mathscr{P}(G)$, i.e.\
  for any $f \in C_{b}(G)$
  \begin{align*}
    \mathcal{L}(X_{k})f = \mathbb{E}[f(X_{k})] \to \pi f, 
    \qquad \text{as } k \to \infty.
  \end{align*}
An elementary fact from the theory of Markov chains
is that, if the Markov operator 
$\mathcal{P}$ is Feller and $\pi$ is a cluster
point of the sequence of measures $(\mu_k) \equiv (\mu_{0}P^k)$ with respect to 
convergence in distribution then $\pi$ is an invariant probability measure \cite[Theorem 1.10]{Hairer2021}.  
We will assume existence of invariant measures.  For more on this theory readers are referred to 
\cite{Hairer2006}, \cite{Hairer2021} and  \cite{Billingsley}.

Quantifying convergence is essential for establishing estimates for the distance 
of the iterates to the limit point, when this exists.  
A sequence $(x_k)$ 
in a metric space $(G,d)$
 is said to \emph{converge R-linearly} to $\tilde{x}$ with rate $c\in [0,1)$
when
\begin{align}\label{eq:R-lin}
\exists \beta>0:\qquad d(x_k,\tilde{x}) \le \beta c^k\quad \forall k\in \Nbb.
\end{align}
The sequence $(x_k)$ 
is said to converge  \emph{Q-linearly} to 
$\tilde{x}$ with rate $c\in [0,1)$ if
\begin{align}\label{eq:Q-lin}
\exists c\in[0,1):\qquad d(x_{k+1},\tilde{x}) \le c d(x_{k},\tilde{x})\quad \forall k\in \Nbb.
\end{align}
By definition, Q-linear convergence implies R-linear convergence with the same rate; 
the converse implication does not hold in general.  
Q-linear convergence is 
encountered with contractive fixed point mappings, and this leads to 
a priori and a posteriori error estimates on the sequence.  This type of convergence is 
referred to as {\em geometric} or {\em exponential} convergence
in different communities.  The crucial distinction between R-linear and Q-linear  
convergence is that R-linear convergence  
permits neither a priori nor a posteriori error estimates.  
For more on these notions see \cite[Chapter 9]{OrtegaRheinboldt70}.

A common metric for spaces of measures is the {\em Wasserstein metric}.  
For $p \geq 1$ let 
  \begin{equation}\label{eq:p-probabiliy measures}
       \mathscr{P}_{p}(G) = \mysetc{\mu \in \mathscr{P}(G)}{ \exists\, x
      \in G \,:\, \int d^{p}(x,y) \mu(\dd{y}) < \infty}.
  \end{equation}
  The Wasserstein $p$-metric on $\mathscr{P}_{p}(G)$, denoted $W_p$, is defined by 
 \begin{equation}\label{eq:Wasserstein}
  W_p(\mu, \nu)\equiv \paren{\inf_{\gamma\in C(\mu, \nu)}\int_{G\times G} 
d^p(x,y)\gamma(dx, dy)}^{1/p}\quad (p\geq 1)
 \end{equation}
 where $C(\mu, \nu)$ is the set of {\em couplings} of $\mu$ and $\nu$:
   \begin{align}
    \label{eq:couplingsDef}
    C(\mu,\nu) := \mysetc{\gamma \in \mathscr{P}(G\times G)}{ \gamma(A
      \times G) = \mu(A), \, \gamma(G\times A) = \nu(A) \quad \forall A
      \in \mathcal{B}(G)}.
  \end{align}
  We will also refer in Proposition \ref{t:sfb} to the {\em Prokhorov-L\'evy 
  distance}, $d_P$, 
  metrizing weak convergence in distribution:
  \begin{equation}\label{eq:PL}
    d_{P}(\mu,\nu) = \inf\mysetc{\epsilon >0}{\mu(A) \le
      \nu(\mathbb{B}(A,\epsilon)) + \epsilon,\, \nu(A)\le
      \mu(\mathbb{B}(A,\epsilon))+\epsilon \quad \forall A \in
      \mathcal{B}(G)}.
  \end{equation} 

 To conclude this section, we state without proof 
 the elementary fact that, for the setting considered 
here, the set of invariant measures is closed. For proof see for example \cite[Section 5]{Hairer2006}.
\begin{lemma}\label{lemma:invMeasuresClosed}
 Let $G$ be a Polish space and let $\mathcal{P}$ be a Feller
  Markov operator, which is in particular the case under Assumption \ref{ass:1}, 
if $T_{i}$ is continuous for all $i\in
  I$. Then the set of associated invariant measures $\inv\mathcal{P}$ is 
closed with respect to the topology of convergence in distribution.
\end{lemma}

\subsection{Regularity}\label{s:regularity}
The regularity of $T_{i}$ depends on the application.  In \cite{HerLukStu22a}
the regularity of $T_{i}$ is used to obtain generic convergence results for the 
corresponding Markov operator, but the regularity of the Markov operator is never 
explicitly defined.  We make this explicit here by lifting the regularity of $T_i$  
to corresponding notions of regularity of Markov operators in 
probability spaces in the case that $T_{i}$ is {\em expansive} following 
the development of such mappings in \cite{LukNguTam18} in a Euclidean
space setting.  When working with expansive mappings, multi-valued mappings appear naturally.  
In the stochastic setting, to ease the notation and avoid certain technicalities,  
we will consider only single-valued mappings $T_i$
that are only almost $\alpha$-firmly nonexpansive {\em in expectation}.
Recent studies define the regularity of fixed point mappings
in {\em $p$-uniformly convex spaces ($p\in(1,\infty)$) with parameter $c>0$} 
(see \cite{BLL} and  \cite{LauLuk21}).  
These are uniquely geodesic metric spaces $(G, d)$ 
for which  the following inequality holds \cite{NaoSil11}:
\begin{equation}\label{e:p-ucvx}
(\forall t\in [0,1])(\forall x,y,z\in G) \quad
d(z, (1-t)x\oplus ty)^p\leq (1-t)d(z,x)^p+td(z,y)^p - \tfrac{c}{2}t(1-t)d(x,y)^p,
\end{equation}
where $w=(1-t)x\oplus ty$ for $t\in (0,1)$ denotes the point $w$ on the geodesic 
connecting $x$ and $y$ such that $d(w,x)=td(x,y)$.
The definitions below hold formally in this nonlinear setting.  The only object whose 
properties are strongly tied to the geometry of the space is the {\em transport discrepancy} 
$\psi$ defined below in \eqref{eq:delta}.  We are limited
so far to locally compact Hadamard spaces, which are locally compact complete CAT($0$) spaces 
(Alexandrov \cite{Alexandrov} and Gromov \cite{Gromov}). 
More generally, a CAT($\kappa$) space is a geodesic metric space with sufficiently small 
triangles possessing comparison triangles with sides the same length as the geodesic 
triangle but for which the distance between points on the geodesic triangle are less than 
or equal to the distance between corresponding points on the comparison triangle.  
CAT($\kappa$) spaces are separable, but not complete, and locally $2$-uniformly convex 
with parameter $c\leq 2$.  A CAT($0$) space has $p=c=2$ in \eqref{e:p-ucvx}.  

The definition below conforms with the same objects defined in \cite{BLL} and  \cite{LauLuk21}.
\begin{definition}[pointwise almost ($\alpha$-firmly) nonexpansive mappings in CAT($0$) spaces]
\label{d:a-fne} 
Let $(G,d)$ be a CAT($0$) metric space and $D\subset G$ and let 
$\mymap{F}{D}{G}$.
  \begin{enumerate}[(i)]
  \item The mapping $F$  is said to be \emph{pointwise  
  almost nonexpansive at $x_0\in D$ on $D$}, abbreviated {\em pointwise ane},
whenever  
\begin{equation}\label{eq:pane}
\exists \epsilon\in[0,1):\quad  d(Fx, Fx_0) \le \sqrt{1+\epsilon}\,d(x,x_0), 
\qquad \forall x \in D.
\end{equation}
The {\em violation} is a value of $\epsilon$ for which \eqref{eq:pane} holds.
    When the above inequality holds for all $x_0\in D$ then $F$ is said to be 
    {\em almost nonexpansive on $D$} ({\em ane}).  When $\epsilon=0$ the 
    mapping $F$ is said to be 
    {\em (pointwise) nonexpansive}.  
  \item The mapping $F$ is said to be {\em pointwise almost $\alpha$-firmly 
nonexpansive  at $x_0\in D$ on $D$}, abbreviated {\em pointwise a$\alpha$-fne} 
  whenever
\begin{eqnarray}
&&
	\exists \epsilon\in[0,1)\mbox{ and }\alpha\in (0,1):\nonumber\\
	  \label{eq:paafne}&&	\quad	d^2(Fx,Fx_0) \le (1+\epsilon) d^2(x,x_0) - 
      \tfrac{1-\alpha}{\alpha}\psi(x,x_0, Fx, Fx_0)\qquad  x \in D,
\end{eqnarray}
where the {\em transport discrepancy} $\psi$ of  $F$ at $x, x_0$, $Fx$ and $Fx_0$
is defined by  
\begin{eqnarray}
&&\!\!\!\!\!\!\!\!\psi(x,x_0, Fx, Fx_0)\equiv \nonumber\\
\label{eq:delta}
&&\!\!\!\!\!\!\!\! d^2(Fx, x)+d^2(Fx_0, x_0) + d^2(Fx, Fx_0) + 
d^2(x, x_0)  - d^2(Fx, x_0)   - d^2(x,Fx_0).
\end{eqnarray}
When the above inequality holds for all $x_0\in D$ then $F$ is said to be 
{\em almost $\alpha$-firmly nonexpansive on $D$}, ({\em a$\alpha$-fne}). 
The {\em violation} is the constant  
$\epsilon$ for which \eqref{eq:paafne} holds.
When $\epsilon=0$ the mapping $F$ is said to be 
    {\em (pointwise) $\alpha$-firmly nonexpansive}, abbreviated {\em (pointwise) $\alpha$-fne}.  
  \end{enumerate}
\end{definition}

Nonexpansive and $\alpha$-firmly nonexpansive mappings have been studied for decades under 
various names and in various settings: a small selection of important works for our results includes
\cite{mann1953mean},  \cite{krasnoselski1955},  \cite{edelstein1966},  \cite{Browder67},  \cite{Gubin67},
 \cite{Bruck73},  \cite{BruckReich77},  \cite{BaiBruRei78}, 
 \cite{GoeRei84},  \cite{Baillon1996},  \cite{AriLeuLop14},  and \cite{RuiLopNic15}.  
The violation $\epsilon$ in \eqref{eq:pane} and \eqref{eq:paafne} is a 
recently introduced feature in the 
analysis of fixed point mappings, first appearing in this form in \cite{LukNguTam18}.  
Many  are familiar with mappings for which \eqref{eq:pane}  holds with $\epsilon<0$ at all $x_0\in G$ , 
i.e. contraction mappings.  In this case, 
the whole technology of pointwise a$\alpha$-fne mappings is not required since an 
appropriate application of Banach's fixed point theorem delivers existence of fixed points and convergence of 
fixed point iterations at a linear rate.  
We will have more to say about this later;  for the moment it suffices to note that the mappings associated 
with one of our target applications are expansive on all neighborhoods of fixed points and we will therefore require 
another property to guarantee convergence.    

Our definition with $\alpha=1/2$ and $\epsilon=0$ is equivalent to the 
definition of {\em firmly contractive} mappings given in \cite[Definition 6]{Browder67}.
In linear spaces the mappings with $\epsilon=0$ are called 
``averaged'' \cite{BaiBruRei78}. 
Ariza-Ruiz, Leu\c{s}tean and L\'opez-Acedo \cite{AriLeuLop14}
defined $\lambda$-firmly nonexpansive operators on subsets $D$ of 
W-hyperbolic spaces, as those operators satisfying 
\begin{equation}\label{e:RLN fne}
\exists \lambda\in(0,1):\quad  d(Fx, Fy)\leq d((1-\lambda)x\oplus \lambda Fx, (1-\lambda)y\oplus \lambda Fy) 
 \quad \forall x,y\in D.
\end{equation}
Another notion of regularity in the context of Hadamard spaces that 
is equivalent to \eqref{e:RLN fne} for an operator $F:H\to H$ and $x,y\in H$ 
uses 
\begin{equation}\label{eq:phi}
 \phi_F(t):=d((1-t)x\oplus tFx,(1-t)y\oplus tFy),\quad \mbox{for }t\in[0,1].
\end{equation}  
In \cite[Chapter 24]{GoeRei84} an operator $F:H\to H$ is called firmly nonexpansive 
whenever $\phi_F$ is nonincreasing on $[0,1]$ (see also \cite[Definition 2.1.13]{Bacak14}). 
It is clear from the definition that $F:D\to D$ satisfies \eqref{e:RLN fne} 
if and only if $\phi_F$ is a nonincreasing function on $[0,1]$ for all 
 $x,y\in D$.  
Banert \cite[Remark pp.658]{Banert} shows that any mapping $F$ satisfying \eqref{e:RLN fne} for all 
$\lambda\in(0,1]$ is $\alpha$-firmly 
nonexpansive with constant $\alpha=1/2$. 

The transport discrepancy $\psi$ is the key to identifying the regularity required for 
convergence of fixed point iterations in metric spaces and the definition makes clear the 
contribution of the geometry of the space.
\begin{lemma}[$\psi$ is nonnegative in CAT($0$) spaces, Proposition 4 of \cite{BLL}]\label{lem:psi_f nonneg}
Let $(G,d)$ be a CAT($0$) metric space and $\mymap{F}{D}{G}$ for $D\subset G$.  
Then the transport discrepancy defined by \eqref{eq:delta} is nonnegative for all $x,y\in D$. 
Moreover, if $F$ is pointwise a$\alpha$-fne at $x_0\in D$ 
with violation $\epsilon$ on $D$, then $F$ is pointwise ane at $x_0$ on $D$ with 
violation at most $\epsilon$.  
\end{lemma}
In CAT($\kappa$) spaces the above statement does not hold.  It is well known, for example,  
that  in a CAT($\kappa$) metric space the projector onto a convex set is 
$\alpha$-fne with $\alpha=1/2$, but it is not nonexpansive \cite{AriLeuLop14}.  
In Hadamard space settings, when a mapping is firmly nonexpansive, it is clear that it 
is also nonexpansive.  This implication was shown in \cite{AriLeuLop14} to be a 
consequence of Busemann convexity 
and does not hold in general metric spaces.  Nevertheless, the 
implication is recovered
for $p$-uniformly convex spaces for {\em pointwise} firmly nonexpansive mappings
{\em at their fixed points}, since the corresponding transport discrepancy $\psi$ 
is nonnegative in this case \cite[Proposition 4(i)]{BLL}.  
It would be interesting to investigate these notions in Busemann spaces where, based
on known extensions of the tools of variational analysis to Banach spaces,  
we conjecture that many of these notions of regularity carry over.  

On normed linear spaces, when $\|\cdot\|$ is the norm induced by the inner product and 
$d(x,y)=\|x-y\|$, the transport discrepancy $\psi$ defined by \eqref{eq:delta} has 
the representation 
\begin{equation} \label{eq:nice ineq}
\psi(x,x_0, Fx, Fx_0) =  \|(x-Fx)-(x_0-Fx_0)\|^2.
\end{equation}
This representation 
shows the connection between our definition and more classical notions. 
Indeed, in a Hilbert space setting $(G, \|\cdot\|)$, 
a mapping $\mymap{F}{D}{G}$ $(D\subset G)$
is pointwise a$\alpha$-fne  at $x_0$ 
with constant $\alpha$ and violation at most $\epsilon$ on $D$ if and only 
if \cite[Proposition 2.1]{LukNguTam18}
 \begin{eqnarray}
&&\|Fx - Fx_0\|^2 \le (1+\epsilon)\|x - x_0\|^2 -  
\tfrac{1-\alpha}{\alpha}\left\|((x-Fx) - (x_0-Fx_0) \right\|^2
\qquad \forall x \in D.\label{eq:paafne2}
\end{eqnarray}

\subsection{Main results}\label{sec:main statements}
We can now state the main theorem concerning  Markov operators $\Pcal$ with
  update function $\Phi(x,i)=T_i(x)$ and transition kernel 
  $p$ given by \eqref{eq:trans kernel} for self mappings 
  $\mymap{T_i}{G}{G}$.
For any $\mu_0\in \mathscr{P}_2(G)$, we denote 
the distributions of the iterates of Algorithm \ref{algo:RFI} by 
$\mu_{k} =  \mu_0 \mathcal{P}^{k} = \mathcal{L}(X_{k})$, 
and we denote  $d_{W_2}\paren{\mu_k,\inv\mathcal{P}}\equiv 
\inf_{\pi'\in\inv\mathcal{P}}W_2\paren{\mu_{k},\, \pi'}$.  
It will be assumed that $\inv\mathcal{P}\neq\emptyset$.  

Convergence is quantified by an implicitly defined gauge function. 
Recall that $\rho:[0,\infty) \to [0,\infty)$ is a \textit{gauge function} if 
$\rho$ is continuous, strictly increasing 
with $\rho(0)=0$, and $\lim_{t\to \infty}\rho(t)=\infty$. 
This is defined in terms of another nonnegative mapping 
$\mymap{\theta_{\tau,\epsilon}}{[0,\infty)}{[0,\infty)}$
with parameters $\tau>0$ and $\epsilon\geq 0$:
\begin{eqnarray}\label{eq:theta}
(i)~ \theta_{\tau,\epsilon}(0)=0, \mbox{ and}\quad (ii)~ 0<\theta_{\tau,\epsilon}(t)<t
~\forall t\in(0,\tbar]\mbox{ for some }\tbar>0.  
\end{eqnarray}
The gauge $\rho$ is given implicitly by 
\begin{equation}\label{eq:gauge}
 \rho\paren{\paren{\frac{(1+\epsilon)t^2-\paren{\theta_{\tau,\epsilon}(t)}^2}{\tau}}^{1/2}}=
 t\quad\iff\quad
 \theta_{\tau,\epsilon}(t) = \paren{(1+\epsilon)t^2 - \tau\paren{\rho^{-1}(t)}^2}^{1/2}
\end{equation}
for $\tau>0$ fixed.  

The rate of convergence of the sequences of measures will be determined by $\theta_{\tau, \epsilon}$,
with parameters $\tau$ and $\epsilon$ given by a certain characterization of the regularity of the Markov 
operator $\Pcal$ (see Definition \ref{d:afne oa}), 
which translates to a second regularity of the Markov operator $\Pcal$, metric subregularity 
(see Definition \ref{d:(str)metric (sub)reg}) via the  gauge $\rho$.  To quantify linear 
or superlinear convergence it suffices to require, in addition to \eqref{eq:theta},
\begin{eqnarray}\label{eq:theta summable}
\sum_{j=1}^\infty\theta_{\tau,\epsilon}^{(j)}(t)<\infty
~\forall t\in(0,\tbar]\mbox{ for some }\tbar>0.  
\end{eqnarray}
In the case of linear convergence 
this becomes 
\[
\rho(t)=r  t\quad\iff\quad  
\theta_{\tau,\epsilon}(t)=\paren{(1+\epsilon)-\frac{\tau}{r^2}}^{1/2}t\qquad 
(r \geq \sqrt{\tfrac{\tau}{(1+\epsilon)}}).
\]  
The conditions in \eqref{eq:theta} in this 
case simplify to $\theta_{\tau,\epsilon}(t)=\gamma t$ where 
\begin{equation}\label{eq:theta linear}
 0< \gamma\equiv 1+\epsilon-\frac{\tau}{r^2}<1\quad\iff\quad 
\sqrt{\tfrac{\tau}{(1+\epsilon)}}\leq  r \leq \sqrt{\tfrac{\tau}{\epsilon}}.
\end{equation}

\begin{thm}[convergence rates]\label{t:msr convergence} 
  Let $(H,d)$ be a separable Hadamard space and let $G\subset H$ be compact.  
  Let $\mymap{T_i}{G}{G}$ be continuous for all $i\in I$ and define
  $\mymap{\Psi}{\mathscr{P}_2(G)}{\mathbb{R}_+}\cup\{+\infty\}$ by
 \begin{equation}\label{eq:Psi}
\Psi(\mu)\equiv \inf_{\pi\in\inv\mathcal{P}}\inf_{\gamma\in C_*(\mu,\pi)}
\left(\int_{G\times G}
\mathbb{E}\left[\psi(x,y, T_\xi x, T_\xi y)\right]\ \gamma(dx, dy)\right)^{1/2}.
 \end{equation}
  Assume furthermore:
  \begin{enumerate}[(a)]
  \item \label{t:msr convergence c}  there is at least one 
  $\pi \in\inv\mathcal{P}\cap \mathscr{P}_2(G)$  where $\mathcal{P}$ is the Markov operator with
  update function $\Phi$ given by \eqref{eq:trans kernel};
  \item\label{t:msr convergence a} $\Phi$ satisfies
  \begin{eqnarray}\label{eq:aafne i.e.}
    &&\exists \epsilon\in [0,1), \alpha\in (0,1):\quad \forall x,y \in G,\\
    &&\quad\mathbb{E}\ecklam{d^2(\Phi(x,\xi),\Phi(y,\xi))}\leq 
    (1+\epsilon)d^2(x,y) - 
    \tfrac{1-\alpha}{\alpha}\mathbb{E}\left[\psi(x,y, \Phi(x,\xi), \Phi(y,\xi))\right],
    \nonumber
  \end{eqnarray}
  and
  \item\label{t:msr convergence b} 
    $\Psi(\pi)=0\iff \pi\in\inv\mathcal{P}$ and for all $\mu\in\mathscr{P}_{2}(G)$ 
  \begin{eqnarray}
    \inf_{\pi\in\inv\mathcal{P}} W_2(\mu, \pi) &=& 
    \inf_{\pi\in\Psi^{-1}(0)}W_2(\mu, \pi)
    \nonumber\\
    &\le& 
    \rho\paren{d_{\mathbb{R}}(0,\Psi(\mu))}
    = \rho(\Psi(\mu)).\nonumber
  \end{eqnarray}
    with gauge $\rho$ given implicitly by \eqref{eq:gauge} with 
    $\tau=(1-\alpha)/\alpha$ and the function $\theta_{\tau,\epsilon}$ satisfying \eqref{eq:theta}
    where $t_0\equiv d_{W_2}\paren{\mu_0, \inv\mathcal{P}\cap\mathscr{P}_2(G)}<\tbar$ for all $\mu_0\in \mathscr{P}_2(G)$.
\end{enumerate}
Then for any $\mu_0\in \mathscr{P}_2(G)$ 
the distributions $(\mu_k)$ of the iterates of Algorithm \ref{algo:RFI} 
satisfy 
\begin{equation}\label{eq:gauge convergence}
d_{W_2}\paren{\mu_{k+1},\inv\mathcal{P}}
\leq \theta_{\tau,\epsilon}\paren{d_{W_2}\paren{\mu_k,\inv\mathcal{P}}} 
\quad \forall k \in \mathbb{N}.
\end{equation}%
If in addition $\theta_{\tau,\epsilon}$ satisfies \eqref{eq:theta summable}, then 
$\mu_k\to \pi^{\mu_0}\in\inv\mathcal{P}\cap\mathscr{P}_2(G)$ as $k\to\infty$ 
in the $W_2$ metric with rate $O(s_k(t_0))$ 
where $s_k(t_0)\equiv\lim_{N\to\infty}\sum_{j=k}^N \theta_{\tau,\epsilon}^{(j)}(t_0)$.
\end{thm}

An immediate corollary of this theorem is the following 
specialization to linear convergence. 
\begin{cor}[linear convergence rates]
\label{t:msr convergence - linear} 
Under the same assumptions as in Theorem \ref{t:msr convergence},
if $\Psi$ satisfies \eqref{t:msr convergence b} 
with gauge $\rho(t)=r \cdot t$ and constant $r $
satisfying $\sqrt{\frac{1-\alpha}{\alpha(1+\epsilon)}}\leq r \
<\sqrt{\frac{1-\alpha}{\alpha\epsilon}}$, 
then the sequence of iterates $(\mu_k)$ converges 
R-linearly to 
some $\pi^{\mu_0}\in\inv\mathcal{P}\cap\mathscr{P}_2(G)$:  
\begin{equation}%
d_{W_2}\paren{\mu_{k+1},\inv\mathcal{P}}
\leq c\, d_{W_2}\paren{\mu_k,\inv\mathcal{P}}
\end{equation}%
where $c\equiv \sqrt{1+\epsilon -\paren{\tfrac{1-\alpha}{r^2\alpha}}}<1$
and $r \geq r'$ satisfies $r\geq\sqrt{(1-\alpha)/\alpha(1+\epsilon)}$.
If $\inv\mathcal{P}$ consists of a single point then convergence is 
Q-linear.  
\end{cor}

\section{Background theory and proofs of the main results.}
As indicated by the discussion following Definition \ref{d:a-fne}, assumption \eqref{t:msr convergence a} 
of Theorem \ref{t:msr convergence} has deep roots in fixed point theory.  Assumption  
\eqref{t:msr convergence b} has a similarly central significance in variational analysis.  
We develop each of these assumptions for the present setting in order. 

\subsection{Almost $\alpha$-firm nonexpansive mappings in expectation}
To begin, we develop assumption \eqref{t:msr convergence a} of Theorem \ref{t:msr convergence}. 
The next definition uses the update function $\Phi$ defined 
in Assumption \ref{ass:1}(b) and is an extension of \cite[Definition 2.8]{HerLukStu22a}.   
\begin{definition}[pointwise almost ($\alpha$-firmly) nonexpansive in
  expectation]\label{d:afne oa}
  Let $(G,d)$ be a CAT($0$) metric space, let $\mymap{T_i}{G}{G}$ for
  $i\in I$, and let 
   $\mymap{\Phi}{G\times I}{G}$ be given by $\Phi(x,i) = T_{i}x$.
  Let $\psi$ be defined by \eqref{eq:delta} and   
let $\xi$ be an $I$-valued random variable.  
  \begin{enumerate}[(i)]
  \item   The mapping $\Phi$ is said to be \emph{pointwise
      almost nonexpansive in expectation at $x_0\in G$} on $G$, abbreviated 
      {\em pointwise ane in expectation}, whenever 
    \begin{align}\label{eq:panee}
		\exists \epsilon\in [0,1):\quad 
		\mathbb{E}\ecklam{d(\Phi(x,\xi), \Phi(x_{0},\xi))} \le \sqrt{1+\epsilon}\,
		d(x,x_0), \qquad \forall x \in G.
    \end{align}
    When the above inequality holds for all $x_0\in G$ then
    $\Phi$ is said to be {\em almost nonexpansive -- ane -- in expectation on
      $G$}.  As before, the violation is a value of  $\epsilon$ for which 
  \eqref{eq:panee} holds. When the violation is $0$, the qualifier ``almost'' is dropped. 
  \item The mapping $\Phi$ is said to be {\em pointwise almost $\alpha$-firmly 
nonexpansive  in expectation at $x_0\in G$} on $G$, abbreviated {\em pointwise a$\alpha$-fne
in expectation},  whenever
  \begin{eqnarray}\label{eq:paafne i.e.}
&&\exists \epsilon\in [0,1), \alpha\in (0,1):\quad \forall x \in G,\\
&&\quad\mathbb{E}\ecklam{d^2(\Phi(x,\xi),\Phi(x_0,\xi))}\leq 
(1+\epsilon)d^2(x,x_0) - 
\tfrac{1-\alpha}{\alpha}\mathbb{E}\left[\psi(x,x_0, \Phi(x,\xi), \Phi(x_0,\xi))\right].
\nonumber
\end{eqnarray}
When the above inequality holds for all $x_0\in G$ then $\Phi$ is said to be 
{\em almost $\alpha$-firmly nonexpansive (a$\alpha$-fne) in expectation on $G$}.  The 
violation is a value of $\epsilon$ for which \eqref{eq:paafne i.e.} holds. 
When the violation is $0$, the qualifier ``almost'' is dropped and the abbreviation 
{\em $\alpha$-fne in expectation} is used. 
  \end{enumerate}
\end{definition}
\begin{prop}\label{r:nonneg psi_Phi}
Let $(G,d)$ be a CAT($0$) space.  The mapping 
$\mymap{\Phi}{G\times I}{G}$ given by $\Phi(x,i) = T_{i}x$
is pointwise a$\alpha$-fne in 
expectation  at $y$ on $G$ with constant $\alpha$ and violation at most $\epsilon$
and pointwise ane in expectation  at $y$ on $G$ 
with violation at most $\epsilon$ 
whenever $T_i$ is pointwise a$\alpha$-fne at $y$ on $G$ with constant $\alpha$ and 
violation no greater than $\epsilon$ for all $i$. 
\end{prop}
\begin{proof}
 By Lemma \ref{lem:psi_f nonneg},  whenever 
$(G,d)$ is a CAT($0$) space 
$\psi(x,y, \Phi(x,i), \Phi(y,i))\geq 0$ for all $i$ and for all $x,y\in G$, so 
the expectation $\mathbb{E}\left[\psi(x,y, \Phi(x,\xi), \Phi(y,\xi))\right]$ 
is well-defined and nonnegative for all $x,y\in G$ (the value $+\infty$ can be attained).  
This implies that, for all $i$, $T_i$ is pointwise ane at $y$ on $G$ with violation at most 
$\epsilon$ on $G$ whenever it is pointwise a$\alpha$-fne at $y$ with constant 
$\alpha$ on $G$ with violation at most 
$\epsilon$ on $G$ for all $i$.  It follows immediately from the definition, then, that  
$\Phi$ is pointwise a$\alpha$-fne in expectation at $y$ with constant 
$\alpha$ on $G$ with violation at most 
$\epsilon$ on $G$, and also pointwise ane in expectation at $y$ on $G$ with violation at most 
$\epsilon$ on $G$.  
\end{proof}

Following \cite[Defintion 2.10]{HerLukStu22a} we lift these notions of regularity to Markov operators.  
Denote the set of couplings where the distance 
$W_2(\mu_1, \mu_2)$ is attained by 
\begin{equation}\label{eq:Gamma opt}
 C_*(\mu_1,\mu_2)\equiv \left\{\gamma\in C(\mu_1, \mu_2)~\big{|}~ 
\int_{G\times G}d^2(x,y)\ \gamma(dx, dy) = W_2^2(\mu_1,\mu_2)\right\}.
\end{equation}
Even though $W_2(\mu_1,\mu_2)$ is defined as the infimum over all couplings, 
whenever this is finite the infimum is attained, and hence in this case $C_*(\mu_1,\mu_2)$
is nonempty \cite[Lemma A.7]{HerLukStu22a}. 

\begin{definition}[pointwise almost ($\alpha$-firmly) nonexpansive Markov operators]\label{d:afne rw}
  Let $(G,d)$ be a CAT($0$) metric space, and let 
$\mathcal{P}$ be a Markov operator with transition kernel
\[
  (x\in G) (A\in
  \mathcal{B}(G)) \qquad p(x,A) \equiv \mathbb{P}(\Phi(x,\xi) \in A)
\]
where $\xi$ is an $I$-valued random variable and 
$\mymap{\Phi}{G\times  I}{G}$ is a measurable update function. 
  Let $\psi$ be defined by \eqref{eq:delta}.
  \begin{enumerate}[(i)]
  \item   The Markov operator is said to be \emph{pointwise
	  almost nonexpansive in measure at $\mu_0\in \mathscr{P}(G)$} on $\mathscr{P}(G)$, 
	  abbreviated {\em pointwise ane in measure}, whenever 
    \begin{align}\label{eq:paneim}
		\exists \epsilon\in [0,1):\quad W_2(\mu\Pcal, \mu_0\Pcal) \le \sqrt{1+\epsilon}\,
		W_2(\mu, \mu_0), \qquad \forall \mu\in \mathscr{P}(G).
    \end{align}
    When the above inequality holds for all $\mu_0\in \mathscr{P}(G)$ then
    $\Pcal$ is said to be {\em almost nonexpansive (ane)  in measure on
      $\mathscr{P}(G)$}.  As before, the violation is a value of  $\epsilon$ for which 
  \eqref{eq:paneim} holds. When the violation is $0$, the qualifier ``almost'' is dropped. 
  \item The Markov operator $\Pcal$ is said to be {\em pointwise almost $\alpha$-firmly 
nonexpansive  in measure at $\mu_0\in G$} on $\mathscr{P}(G)$, 
abbreviated {\em pointwise a$\alpha$-fne in measure}, 
  whenever
  \begin{eqnarray}
&&\exists \epsilon\in [0,1), \alpha\in (0,1): 
\qquad \forall\mu\in \mathscr{P}(G),\forall \gamma\in C_*(\mu, \mu_0)
\nonumber\\
&& W_2(\mu\Pcal, \mu_0\Pcal)^2\leq 
(1+\epsilon)W_2(\mu, \mu_0)^2 - \nonumber\\
&&\qquad \qquad\qquad \qquad
\tfrac{1-\alpha}{\alpha}\int_{G\times G}\mathbb{E}\left[\psi(x,y, \Phi(x,\xi), \Phi(y,\xi))\right] \gamma(dx, dy).
	  \label{eq:paafne i.m.}
\end{eqnarray}
When the above inequality holds for all $\mu_0\in \mathscr{P}(G)$ then $\Pcal$ is said to be 
{\em a$\alpha$-fne in measure on $\mathscr{P}(G)$}.  The 
violation is a value of $\epsilon$ for which \eqref{eq:paafne i.m.} holds. 
When the violation is $0$, the qualifier ``almost'' is dropped and the abbreviation 
{\em $\alpha$-fne in measure} is employed. 
  \end{enumerate}
\end{definition}

\begin{rem}\label{r:nonneg psi_Phi2}
By Lemma \ref{lem:psi_f nonneg},  when $(G, d)$ is a CAT($0$) space 
 the expectation on the right hand side of \eqref{eq:paafne i.m.} is nonnegative, and 
 the corresponding Markov operator is pointwise ane in measure 
 at $\mu_0$ whenever it is pointwise a$\alpha$-fne in measure at $\mu_0$
 (Proposition \ref{r:nonneg psi_Phi}). 
 In particular, when  
$\mu=\mu_0\in \inv\mathcal{P}$ the left hand side is zero and   
\[
\int_{G\times G} 
\mathbb{E}\left[\psi(x,y, T_\xi x, T_\xi y)\right]\ \gamma(dx, dy) = 0.
\]
Here the optimal coupling is the diagonal of the product space $G\times G$
and $\psi(x,x, T_\xi x, T_\xi x)=0$ for all $x\in G$.  
\end{rem}

\begin{prop}\label{thm:Tafne in exp 2 pafne of P}
  Let $(G,d)$ be a separable Hadamard space (a complete CAT($0$) space) and 
  $\mymap{T_i}{G}{G}$ for $i\in I$, let 
   $\mymap{\Phi}{G\times I}{G}$ be given by $\Phi(x,i) = T_{i}x$
   and let $\psi$ be defined by \eqref{eq:delta}. 
   Denote  by $\mathcal{P}$  the Markov operator with update function $\Phi$ and 
   transition kernel $p$ defined by
  \eqref{eq:trans kernel}.
  If $~\Phi$  is a$\alpha$-fne in expectation on $G$ with
  constant $\alpha\in (0,1)$ and violation $\epsilon\in [0,1)$, then the Markov operator 
  $\mathcal{P}$ is a$\alpha$-fne 
  in measure on $\mathscr{P}_2(G)$ 
  with constant $\alpha$ and violation at most $\epsilon$, that is, $\Pcal$ 
  satisfies 
  \begin{eqnarray}
W_2^{2}(\mu_1\mathcal{P}, \mu_2\mathcal{P})&\le&    
   (1+\epsilon)W_2^{2}(\mu_1, \mu_2) -  
   \tfrac{1-\alpha}{\alpha}  \int_{G\times G}
   \mathbb{E}\ecklam{\psi(x,y, \Phi(x,\xi), \Phi(y,\xi))}\ \gamma(dx, dy)
   \nonumber\\
\label{eq:alphfne meas}&&   
\qquad\qquad\qquad\qquad   
 \forall \mu_2, \mu_1\in \mathscr{P}_2(G), \ \forall \gamma\in C_*(\mu_1,\mu_2).  
  \end{eqnarray}
\end{prop}
\begin{proof} If $W_2(\mu_1,\mu_2)=\infty$ the inequality holds trivially 
with the convention 
$+\infty-(+\infty)= +\infty$.    
So consider the case where $W_2(\mu_1,\mu_2)$ is finite.  Since $(G, d)$ is a separable, 
complete metric 
space, by \cite[Lemma A.7]{HerLukStu22a}
the set of  optimal couplings $C_*(\mu_1, \mu_2)$ is nonempty.  
Since $\Phi$ is a$\alpha$-fne in expectation on $G$ with 
constant $\alpha$ and violation $\epsilon$, we have
\begin{eqnarray*}
&&\int_{G\times G}\mathbb{E}\left[d^2(\Phi(x, \xi),\Phi(y, \xi))\right]\ \hat{\gamma}(dx, dy)
\leq \\
&&\qquad\qquad\qquad\qquad 
\int_{G\times G} \paren{(1+\epsilon) d^2(x,y)    - \tfrac{1-\alpha}{\alpha} 
\mathbb{E}\left[\psi_c(x,y, \Phi(x, \xi), \Phi(y, \xi))\right]}
\ \hat{\gamma}(dx, dy),
  \end{eqnarray*}
  where $\hat{\gamma}$ is any coupling in $ C(\mu_1,\mu_2)$, not 
  necessarily optimal.  In particular, 
  since, for a random variable $X\sim \mu_1$, we have $\Phi(X, \xi)\sim \mu_1\mathcal{P}$,
  and for a random variable $Y\sim \mu_2$, 
  we have $\Phi(Y, \xi)\sim \mu_2\mathcal{P}$, 
  then, again for any optimal coupling $\gamma\in C_*(\mu_1,\mu_2)$,
   \begin{eqnarray*}
   W_2^2(\mu_1\mathcal{P},\mu_2\mathcal{P})&\leq& 
   \int_{G\times G}\mathbb{E}\left[d^2(\Phi(x, \xi),\Phi(y, \xi))\right]\ \gamma(dx, dy)
  \nonumber\\
  &\leq& 
\int_{G\times G} \paren{(1+\epsilon) d^2(x,y)    - \tfrac{1-\alpha}{\alpha} 
\mathbb{E}\left[\psi_c(x,y, \Phi(x, \xi), \Phi(y, \xi))\right]}
\ \gamma(dx, dy)\nonumber \\
&=&
(1+\epsilon) W_2^2(\mu_1,\mu_2) - \int_{G\times G}\tfrac{1-\alpha}{\alpha} 
\mathbb{E}\left[\psi_c(x,y, \Phi(x, \xi), \Phi(y, \xi))\right]\ \gamma(dx, dy).
  \end{eqnarray*}
  Since the measures $\mu_2, \mu_1\in\mathscr{P}_2(G)$ were arbitrary, as was the 
  optimal coupling $\gamma\in C_*(\mu_1,\mu_2)$, this completes the proof.  
\end{proof}

\subsection{Metric subregularity}
We move now to assumption \eqref{t:msr convergence b} of Theorem \ref{t:msr convergence}.  
In \cite{LukNguTam18} a general quantitative analysis for iterations of 
expansive
fixed point mappings is proposed consisting of two principle components: 
the constituent mappings are pointwise a$\alpha$-fne, and the transport 
discrepancy of the fixed point operator is 
{\em metrically subregular}.  
Recall that, for any mapping $\mymap{\Psi}{A}{B}$,
the inverse mapping $\Psi^{-1}(y)\equiv \set{z\in A}{\Psi(z)=y}$, which 
clearly can be set-valued. 
\begin{definition}[metric subregularity]\label{d:(str)metric (sub)reg}
$~$ Let $(A, d_A)$ and $(B,d_B)$ be metric spaces and let $\mymap{\Psi}{A}{B}$.
The mapping $\Psi$ is called \emph{metrically subregular with respect to the metric $d_B$ 
for $y\in B$ relative 
to $\Lambda\subset A$ on $U\subset A$ with gauge $\rho$} whenever
\begin{equation}\label{e:metricregularity}
\inf_{z\in \Psi^{-1}(y)\cap \Lambda} d_A\paren{x, z}\leq \rho( d_B\paren{y, \Psi(x)})
\quad \forall x\in U\cap \Lambda.
\end{equation}
\end{definition}

Our definition is modeled after \cite{DonRoc14}, where the case where the 
gauge is just a linear function -- $\rho(t)=r t$ for $r>0$ -- is developed.  
In this case, metric subregularity
is one-sided Lipschitz continuity of the (set-valued) inverse mapping $\Psi^{-1}$.  
We will refer to the case when the gauge is linear as {\em linear metric subregularity}.
For connections of this notion to the concept of {\em transversality} in differential 
geometry and its use in variational analysis see \cite{Ioffe17}.  Metric subregularity 
is also related to o-minimal geometry methods and generalizes what are known as 
{\L}ojasiewicz inequalities \cite{BolDanLeyMaz10}.

In the case of 
a finite dimensional linear transition kernel, the constant $\rho$ of metric subregularity 
characterizes the {\em spectral gap}, or the difference between the two  
largest eigenvalues of $\Id-T_i$.  For another example derived from \cite[Propositions 9-10]{LukTebTha18},
the alternating projections operator 
for finding the intersection of two affine subspaces (assumed nonempty) has a transport 
discrepancy mapping $\psi$ whose constant of metric subregularity grows to infinity 
as the (Friedrichs) angle between the sets goes to zero;  in other words, the 
smaller the angle between the subspaces, the slower the rate of convergence of alternating 
projections.

We apply metric regularity to the Markov operator on $\mathscr{P}(G)$ 
with the Wasserstein metric.  In particular, 
the gauge of metric subregularity $\rho$  is constructed 
implicitly in \eqref{eq:gauge} from $\mymap{\theta_{\tau,\epsilon}}{[0,\infty)}{[0,\infty)}$ satisfying 
\eqref{eq:theta}.

Metric subregularity plays a central role in the implicit function paradigm for solution 
mappings (see \cite{CANO2},  \cite{DonRoc14});  it is also notoriously difficult to verify.  
Linear metric subregularity was shown in \cite[Theorem 3.15]{HerLukStu19a} to be {\em necessary} 
 and sufficient for R-linear convergence in expectation of random function 
 iterations for consistent stochastic feasibility.  This result is a stochastic
 analog of \cite[Theorem 2]{LukTebTha18} in the deterministic setting and all of this 
 has been extended to nonlinear spaces in \cite[Theorem 16]{LauLuk21}.   It is 
 an open problem whether metric subregularity is necessary in the present setting, though 
 we see no reason why it should not be.  
 
 We apply this to the Markov operator $\mathcal{P}$ on the metric space 
 $(\mathscr{P}_2(G), W_2)$ in the following manner.  Recall the transport discrepancy $\psi$ 
of a self-mapping $T_i$ defined in \eqref{eq:delta}. 
 We construct the surrogate mapping 
$\mymap{\Psi}{\mathscr{P}(G)}{\mathbb{R}_+}\cup\{+\infty\}$ 
 defined by \eqref{eq:Psi}.
 We call this the {\em invariant Markov transport discrepancy} corresponding to the mappings $T_i$.
 It is not guaranteed  that both $\inv\mathcal{P}$ and $C_*(\mu,\pi)$ are nonempty;  
 when at least one of these is empty, we define $\Psi(\mu)\equiv +\infty$. 
It is clear that $\Psi(\pi)=0$ for any 
$\pi \in \inv\Pcal$.  Whether $\Psi(\mu)=0$ only when $\mu\in \inv\Pcal$ is a property
of the {\em space} $(G, d)$.  Indeed, as noted in the discussion after 
Lemma \ref{lem:psi_f nonneg}, in CAT($\kappa$) spaces with $\kappa>0$
the transport discrepancy $\psi$ can be negative, and so 
by cancellation it could happen on such spaces that the invariant Markov transport discrepancy
$\Psi(\mu)=0$ for $\mu\notin \inv\Pcal$. 
In CAT($0$) spaces, and hence Hadamard spaces, $\psi$ is nonnegative but this is still 
not enough to guarantee that the invariant Markov transport discrepancy $\Psi$ takes the value 
$0$ at $\mu$ if and only if 
$\mu\in\inv\Pcal$.  The regularity we require of $\Pcal$ is that 
the invariant Markov transport discrepancy $\Psi$ takes the value $0$ at $\mu$ if and only if 
$\mu\in\inv\Pcal$, and is metrically subregular for $0$ relative to $\mathscr{P}_2(G)$ 
on $\mathscr{P}_2(G)$ 
defined in \eqref{eq:p-probabiliy measures}.  

\subsection{Contractive Markov operators}
Before moving to our main results, we put the more familiar contractive 
mappings into the present context.  
A survey of random function iterations for contractive mappings 
in expectation can be found in \cite{Stenflo2012}.     An immediate
consequence of \cite[Theorem 1]{Stenflo2012} is the existence of a
unique invariant measure and linear convergence in the Wasserstein
metric from any initial distribution to the invariant measure.  

Contraction Markov operators have been studied in \cite{Oll09} and \cite{JouOll10} using the 
parallel notion of the {\em coarse Ricci curvature} $\kappa(x, y)$ of the Markov 
operator $\mathcal{P}$ between two points 
$x$ and $y$:
\[
\kappa(x, y)\equiv 1 - \frac{W_1(\delta_x\Pcal, \delta_y\Pcal)}{d(x,y)}.
\]
Generalizing this definition to $W_p$ yields the coarse Ricci curvature with respect to $W_p$:
\[
\kappa_p(x, y)\equiv 1 - \frac{W_p^p(\delta_x\Pcal, \delta_y\Pcal)}{d(x,y)^p}.
\]
A few steps lead from this object for the 
Markov operator $\Pcal$ with  update function $\Phi(\cdot, \xi)=T_\xi$ 
and transition kernel defined by \eqref{eq:trans kernel}
to the violation $\epsilon$  in Proposition \ref{thm:Tafne in exp 2 pafne of P}.  
Indeed,  a formal adjustment of the proof of \cite[Proposition 2]{Oll09} establishes that the property
$\kappa_2(x,y)\geq \kappa \in \mathbb{R}$ for all $x, y\in G$ is equivalent to 
\[
W_2(\mu\Pcal, \mu'\Pcal)\leq \sqrt{1-\kappa}\,W_2(\mu, \mu')\quad\forall \mu, \mu'\in \mathscr{P}_2(G).
\]
When $\kappa>0$, i.e. when the coarse Ricci curvature is bounded below by a positive number, this 
characterizes contractivity of the Markov operator.  The negative of the violation in \eqref{eq:paneim}
is just a lower bound on the coarse Ricci curvature in $W_2$: 
$-\epsilon = \kappa\leq \kappa_2(x,y)$ for all $x,y\in G$.  
The consequences of Markov operators with Ricci curvature bounded below by a positive number 
have been extensively investigated.  Error estimates for 
Markov chain Monte Carlo methods under the assumption of positive Ricci curvature in $W_1$ 
(i.e. {\em negative} violation) are explored in \cite{JouOll10}.  Applications
to waiting queues, the Ornstein–Uhlenbeck process on $\Rn$ and Brownian
motion on positively curved manifolds, as well as demonstrations of how to verify the 
assumptions on the Ricci curvature are developed in \cite{Oll09}.  
Our approach extends this to {\em expansive} mappings, which 
allows one to treat our target application 
of electron density reconstructions from X-FEL experiments 
(see Section \ref{sec:incFeas}).  

The next result shows that update functions $\Phi$ that are contractions in expectation 
generate $\alpha$-fne Markov operators with metrically 
subregular invariant Markov transport discrepancy.  In this context see \cite{Klo22}.
The first two parts of the statement are 
the content of \cite[Proposition 4.4]{HerLukStu22a}.
\begin{thm}
\label{thm:contraInExpec}
   Let $(G,\|\cdot\|)$ be a Hilbert space, let 
   $\mymap{T_i}{G}{G}$ for $i\in I$ and let 
  $\mymap{\Phi}{G \times I }{G}$ be given by
  $\Phi(x,i)\equiv T_i(x)$.  
  Denote by  $\mathcal{P}$ the Markov
  operator with  update function $\Phi$ and transition kernel $p$ defined by 
  \eqref{eq:trans kernel}.  
  Suppose that $\Phi$ is a 
  \emph{contraction in expectation} with 
  constant $r<1$, i.e.\  
  $\mathbb{E}[\|\Phi(x,\xi) - \Phi(y,\xi)\|^2] \le r^2 \|x-y\|^2$ for all $x,y \in
  G$. 
  Suppose in addition  that  
  there exists $y \in G$ with
  $\mathbb{E}[\|\Phi(y,\xi) - y\|^2] < \infty$.  
Then the following hold. 
\begin{enumerate}[(i)]
 \item\label{thm:contraInExpec i} There exists a unique
  invariant measure $\pi \in \mathscr{P}_{2}(G)$ for $\mathcal{P}$ and 
  \begin{align*}
    W_2(\mu_0 \mathcal{P}^{n} , \pi) \le r^{n} W_2(\mu_0,\pi)
  \end{align*}
  for all $\mu_0 \in \mathscr{P}_{2}(G)$; that is, the sequence $(\mu_k)$ defined 
  by $\mu_{k+1}=\mu_k\mathcal{P}$ converges  to $\pi$ Q-linearly (geometrically) from any initial 
  measure $\mu_0\in \mathscr{P}_{2}(G)$.
\item\label{thm:contraInExpec ii}   $\Phi$ is $\alpha$-fne 
in expectation with constant $\alpha = (1+r)/2$, and the Markov operator $\mathcal{P}$ 
is $\alpha$-fne on $\mathscr{P}_{2}(G)$; that is,  $\mathcal{P}$
satisfies \eqref{eq:alphfne meas} 
with $\epsilon=0$ and constant $\alpha = (1+r)/2$ on $\mathscr{P}_{2}(G)$.
\item\label{thm:contraInExpec iii} If $\Psi$  defined by \eqref{eq:Psi} satisfies 
\begin{equation}\label{e:Hood}
\exists q>0:\quad \Psi(\mu) \geq qW_2(\mu\Pcal, \mu)\quad\forall \mu\in\mathscr{P}_{2}(G),
\end{equation}
then 
$\Psi$ is linearly metrically subregular for $0$ relative to 
$\mathscr{P}_{2}(G)$ on $\mathscr{P}_{2}(G)$ with gauge $\rho(t) = (q(1-r))^{-1} t$.
\end{enumerate}
\end{thm}
\begin{proof}
Parts \eqref{thm:contraInExpec i} and \eqref{thm:contraInExpec ii} were proved in 
\cite[Proposition 4.4]{HerLukStu22a}.

The proof of \eqref{thm:contraInExpec iii}  is modeled after the proof of \cite[Theorem 32]{BLL}. 
By the triangle inequality and part \eqref{thm:contraInExpec i} we have 
\begin{eqnarray}
W_2(\mu_{k+1},\mu_k) &\geq& W_2(\mu_k, \pi) - W_2(\mu_{k+1}, \pi) \nonumber\\
&\geq& (1-r)W_2(\mu_k, \pi) \quad  \forall k\in \Nbb.
\label{e:Robin}
\end{eqnarray}
On the other hand, \eqref{e:Hood} implies that $\Psi$ takes the value zero only at invariant 
measures 
so that by    
the uniqueness of invariant measures established in part \eqref{thm:contraInExpec i}
\[
\Psi^{-1}(0)\cap \mathscr{P}_{2}(G) = \inv \Pcal\cap \mathscr{P}_{2}(G) = \{\pi\}.
\]
Combining this with \eqref{e:Robin} and \eqref{e:Hood} then yields for all $k\in \Nbb$
\begin{eqnarray*}
|\Psi(\mu_k)-0| = \Psi(\mu_k)&\geq& qW_2(\mu_{k+1},\mu_k)\nonumber\\
&\geq& q(1-r)W_2(\mu_k, \Psi^{-1}(0)\cap \mathscr{P}_{2}(G)).  
 \end{eqnarray*}
In other words, 
 \begin{equation}\label{e:dumber}
(q(1 - r))^{-1}|\Psi(\mu_k)-0|\geq W_2(\mu_k, \Psi^{-1}(0)\cap \mathscr{P}_{2}(G)) 
\quad  \forall k\in \Nbb.
 \end{equation}
Since this holds for {\em any} sequence$(\mu_k)$ 
 initialized with 
$\mu_0\in \mathscr{P}_{2}(G)$,  
we conclude that $\Psi$ is metrically subregular for $0$ relative to $\mathscr{P}_{2}(G)$ with 
gauge $\rho(t)=(q(1 - r))^{-1}t$ on $\mathscr{P}_{2}(G)$, as claimed.  
\end{proof}

\subsection{Proofs of the main results}
\label{sec:mainres}
We are now in a position to prove the main result.\\

\noindent {\em Proof of Theorem \ref{t:msr convergence}.} 
  First note that by assumption \eqref{thm:contraInExpec i} the Markov operator 
  $\mathcal{P}$ is a
self-mapping on $\mathscr{P}_{2}(G)$, hence
$W_2(\mu,\mu\mathcal{P}) < \infty$, and for any 
$\mu_1, \mu_2\in \mathscr{P}_{2}(G)$ the set of optimal couplings
$C_*(\mu_1, \mu_2)$ is nonempty \cite[Lemma A.7]{HerLukStu22a}.  
Since $(H, d)$ is 
a Hadamard space and $G\subset H$, the function $\Psi(\mu)$ defined by 
\eqref{eq:Psi}
is extended real-valued, nonnegative (see Lemma \ref{lem:psi_f nonneg}), 
and finite since $C_*(\mu, \pi)$ and 
$\inv\mathcal{P}$ are nonempty.
Moreover, by assumption 
\eqref{t:msr convergence b} $\Psi^{-1}(0)=\inv\mathcal{P}$
 and 
\begin{equation}\label{eq:rate step 1}
   \tfrac{1-\alpha}{\alpha}
   \paren{\rho^{-1}\paren{%
   \inf_{\pi\in\inv\mathcal{P}}W_2(\mu, \pi)}}^{2}
   \leq \tfrac{1-\alpha}{\alpha}\Psi^2(\mu).
\end{equation}
On the other hand, note that assumption \eqref{t:msr convergence a} is that 
$\Psi$ is a$\alpha$-fne in expectation, so 
 by definition \eqref{eq:Psi}, assumption 
\eqref{t:msr convergence a} and 
Proposition \ref{thm:Tafne in exp 2 pafne of P} 
(which applies because we are on a separable Hadamard space) we
have 
\begin{eqnarray}
  \tfrac{1-\alpha}{\alpha}\Psi^2(\mu)
   &\leq& 
   \int_{G\times G}\mathbb{E} \left[\psi(x,y, T_\xi x, T_\xi y)\right]\ \gamma(dx, dy)
   \nonumber\\
   &\leq&
     (1+\epsilon) W_2^{2}(\mu,\pi) - W_2^{2}(\mu\mathcal{P},\pi)
      \quad\forall \pi\in\inv\mathcal{P}, 
      \forall \mu\in\mathscr{P}_{2}(G).
      \label{eq:rate step 2}
\end{eqnarray}
Incorporating \eqref{eq:rate step 1} into 
\eqref{eq:rate step 2} and rearranging the inequality yields
\begin{eqnarray*}
W_2^{2}(\mu\mathcal{P},\pi)
   \!&\leq&\! 
      (1+\epsilon) W_2^{2}(\mu,\pi) - 
         \tfrac{1-\alpha}{\alpha}
   \paren{\rho^{-1}\paren{%
   \inf_{\pi'\in\inv\mathcal{P}}W_2(\mu, \pi')}}^{2}
      \quad\forall \pi\in\inv\mathcal{P}, 
      \forall \mu\in\mathscr{P}_{2}(G).
\end{eqnarray*}
Since this holds at {\em any} $\mu\in\mathscr{P}_2(G)$, it certainly 
holds at the iterates $\mu_k$ with initial distribution $\mu_0\in\mathscr{P}_2(G)$
since $\mathcal{P}$ is a 
self-mapping on $\mathscr{P}_2(G)$.   Therefore 
\begin{eqnarray}\label{eq:gauge convergence 0}
&& W_2\paren{\mu_{k+1},\, \pi}
\leq\\ 
&&\qquad \sqrt{(1+\epsilon) W_2^2\paren{\mu_{k},\, \pi} - 
\frac{1-\alpha}{\alpha}
\paren{\rho^{-1}\paren{%
\inf_{\pi'\in\inv\mathcal{P}}W_2\paren{\mu_{k},\, \pi'}
}
}^2
} 
\quad\forall \pi\in\inv\mathcal{P}, ~
\forall k \in \mathbb{N}.\nonumber
\end{eqnarray}

Equation \eqref{eq:gauge convergence 0} simplifies.  
Indeed, by Lemma \ref{lemma:invMeasuresClosed}, $\inv \mathcal{P}$ is closed 
with respect to convergence 
in distribution.  
Moreover, since $G$ is assumed to be compact, 
$\mathscr{P}_2(G)$ is locally compact (\cite[Remark 7.19]{AmbGigSav2005}
so, for every $k\in\Nbb$  the infimum in  \eqref{eq:gauge convergence 0}
is attained  at some $\pi_k$. 
This yields
\begin{equation}
W_2^{2}(\mu_{k+1},\pi_{k+1} ) \leq 
W_2^{2}(\mu_{k+1},\pi_{k} ) \leq 
(1+\epsilon) W_2^{2}(\mu_k, \pi_k)  - 
\tfrac{1-\alpha}{\alpha}\paren{\rho^{-1}
\paren{ W_2(\mu_k, \pi_k)}}^2
\quad\forall k \in \mathbb{N}.
\label{eq:gauge convergence intermed}
\end{equation}
Taking the square root and recalling \eqref{eq:theta} and \eqref{eq:gauge}
yields \eqref{eq:gauge convergence}.

To obtain convergence, note that for $\mu_0\in\mathscr{P}_{2}(G)$
satisfying $W_2(\mu_0,\pi)<\infty$ and  $\mu_0\mathcal{P}\in\mathscr{P}_{2}(G)$
(exists by compactness of $G$
), the triangle inequality and  
\eqref{eq:gauge convergence intermed} yield
\begin{eqnarray*}
 W_2(\mu_{k+1},\mu_k )&\leq&  W_2(\mu_{k+1},\pi_{k} )
 + W_2(\mu_{k},\pi_{k} )\\
 &\leq &\theta_{\tau,\epsilon}\paren{W_2\paren{\mu_k, \pi_k}}+ W_2(\mu_{k},\pi_{k} ).
\end{eqnarray*}
Using \eqref{eq:gauge convergence} 
and continuing by backwards induction yields
\[
 W_2(\mu_{k+1},\mu_k )\leq \theta_{\tau,\epsilon}^{k+1}\paren{ d_0}+ 
 \theta_{\tau,\epsilon}^k\paren{ d_0}
\]
where $d_0\equiv \inf_{\pi\in\inv\mathcal{P}} W_2\paren{\mu_0,\pi}$.
Repeating this argument, for any $k<m$ 
\[
 W_2(\mu_{m},\mu_k )\leq \theta_{\tau,\epsilon}^{m}\paren{ d_0} + 
 2\sum_{j=k+1}^{m-1}\theta_{\tau,\epsilon}^j\paren{ d_0}
 +\theta_{\tau,\epsilon}^k\paren{ d_0}.
\]
By assumption, $\theta_{\tau,\epsilon}$ satisfies \eqref{eq:theta summable}, so for any $\delta>0$
\begin{eqnarray*}
 W_2(\mu_{m},\mu_k )&\leq& 
 \theta_{\tau,\epsilon}^{m}\paren{ d_0} + 
 2\sum_{j=k+1}^{m-1}\theta_{\tau,\epsilon}^j\paren{ d_0}
 +\theta_{\tau,\epsilon}^k\paren{ d_0}\nonumber\\
& \leq& 2\sum_{j=k+1}^{\infty}\theta_{\tau,\epsilon}^j\paren{ d_0}
+ \theta_{\tau,\epsilon}^k\paren{ d_0}
 <\delta
\end{eqnarray*}
for all $k, m$ large enough;  that is the sequence $(\mu_k)$ 
 is a Cauchy sequence
in $(\mathscr{P}_{2}(G), W_2)$ -- a separable complete metric space 
\cite[Theorem 6.9]{Villani2008} -- and therefore convergent to some probability measure 
$\pi^{\mu_0}\in \mathscr{P}_{2}(G)$. The Markov operator $\mathcal{P}$ is Feller since $T_i$ is continuous
and by \cite[Proposition 3.1]{HerLukStu22a} (see also \cite[Theorem 1.10]{Hairer2021}) when a Feller Markov chain converges
in distribution, it does so to an invariant measure:  $\pi^{\mu_0} \in \inv \mathcal{P}$.
\hfill$\Box$

\begin{rem}\label{r:G compact}
The compactness assumption on $G$ can be dropped if 
$(H,d)$ is a Euclidean space.  
\end{rem}

\noindent{\em Proof of Corollary \ref{t:msr convergence - linear}.}
In the case that the gauge $\rho$ is linear with constant $r'$, 
then $\theta_{\tau,\epsilon}(t)$ is linear with constant 
\[
c=\sqrt{1+ \epsilon  - \frac{1-\alpha}{r^2\alpha}}<1,  
\]
where $r\geq r'$ satisfies $r^2\geq (1-\alpha)/(\alpha(1+\epsilon))$.
Specializing the argument in the proof above to this particular $\theta_{\tau,\epsilon}$ 
shows that, for any
$k$ and $m$ with $k<m$, we have
\begin{equation}\label{DR.3}
    \begin{aligned}
W_2\paren{\mu_m,\mu_{k}} &\,\le\, 
d_0 c^m + 2d_0\sum_{j=k+1}^{m-1}c^j+ d_0c^k.\\
    \end{aligned}
\end{equation}
Letting $m\to \infty$ in \eqref{DR.3} yields R-linear convergence
\eqref{eq:R-lin}
with rate $c$ given above and leading constant 
$\beta= \frac{1+c}{1-c}d_0$.

If, in addition, $\inv\mathcal{P}$ is a singleton, then 
$\{\pi^{\mu_0}\} = \inv\mathcal{P}$
in the above and convergence is actually Q-linear, which completes the proof.
\hfill$\Box$

%

\section{Examples: Stochastic Optimization and Inconsistent Nonconvex Feasibility}
\label{sec:incFeas}

To fix our attention we focus on the following optimization problem
  \begin{equation}\label{eq:opt prob}
  \underset{\mu\in \mathscr{P}_2(\Rn)}{\mbox{minimize}}
  \int_{\Rn}\mathbb{E}_{\xi}[f_{\xi^f}(x) + g_{\xi^g}(x)]\mu(dx). 
  \end{equation}
It is assumed throughout that $\mymap{f_{i}}{\mathbb{R}^{n}}{\mathbb{R}}$ is continuously 
  differentiable for all $i \in I_f$ and that $\mymap{g_{i}}{\mathbb{R}^{n}}{\mathbb{R}\cup{+\infty}}$
  is proper (not everywhere infinite), lower semicontinuous ($g_i(\xbar)\leq \liminf_{x\to\xbar} g(x)$ for 
  all $\xbar$) for all $i \in I_g$ and {\em subdifferentially regular}, i.e. the limiting subdifferential and 
  the regular subdifferential are the same, where the regular subdifferential of $g_i$ at $\xbar$, 
  denoted $\partial g_i(\xbar)$, is defined by 
  \[
   \partial g_i(\xbar)\equiv \set{v}{g_i(x)\geq g_i(\xbar)+\ip{v}{x-\xbar}+o(\|x-\xbar\|)}.
  \]
  In particular, when $g_i$ is differentiable at $\xbar$, the subdifferential is a singleton
  consisting of the gradient:  $\partial g_i(\xbar)=\{\nabla g_i(\xbar)\}$.  
For more background on nonsmooth analysis see \cite{VA}.
The random variable with values on $I_f\times I_g$ will be denoted $\xi = (\xi^f, \xi^g)$. 
This model covers deterministic composite optimization as a special case: $I_f$ and $I_g$ consist
of single elements and the measure $\mu$ is a point mass.  

The algorithms reviewed in this section rely on resolvents of the subdifferentials/gradients of the functions
$f_i$ and $g_i$, denoted $\Jcal_{{\partial f_i},\lambda}$ and $\Jcal_{{\partial g_i},\lambda}$.   The resolvent of a 
multi-valued mapping $F$ from elements in 
$G\subset\Rn$ to one or more elements in ${\mathbb{R}^n}$ 
is defined by 
\[\Jcal_{F, \lambda}(x)\equiv \paren{\tfrac{1}{\lambda}\Id+F}^{-1}(x)
\equiv \set{z\in \Rn}{x=\tfrac{1}{\lambda}z+ F(z)}.\]  
For proper, lower semicontinuous convex functions $\mymap{f}{\Rn}{\mathbb{R}\cup\{+\infty\}}$, 
this is equivalent to the proximal mapping \cite{Moreau65} (often just called the {\em prox mapping})
defined by
\begin{equation}\label{eq:prox}
\prox_{f,\lambda}(x)\equiv\argmin_{y}\{f(y)+ \tfrac{1}{2\lambda}d(y,x)^2\}.
\end{equation}
In general one has
\begin{equation}\label{eq:resolvent}
 \prox_{f,\lambda}(x) \subset \Jcal_{\partial f,\lambda}(x)
\end{equation}
whenever the subdifferential is defined.  When $\lambda=1$ in the definitions above, we will just write 
$\prox_f$ or $\Jcal_{\partial f}$.  
\subsection{Stochastic (nonconvex) forward-backward splitting}\label{ex:spg ncvx}
A splitting algorithm applied to problem \eqref{eq:opt prob} is any algorithm that proceeds by
taking steps with respect to $f_{\xi^f}(x)$ and $g_{\xi^g}(x)$ separately and combining these
either through convex combinations or compositions of some sort.  Classical examples of splitting 
methods are the Gauss-Seidel and Jacobi algorithms for solving linear systems. 
We present a general prescription of what is known as the forward-backward splitting algorithm 
together with abstract properties of the corresponding fixed point mapping, and then specialize 
this to more concrete instances. Algorithm \ref{algo:sfb} is called {\em forward-backward} because the 
step in the direction of the negative gradient $-{t}\nabla f_{\xi^f_k}(X_{k})$ is interpreted
as a ``forward'' step (in the context of differential equations, this would be 
an explicit Euler step), while the step taken by applying the resolvent $\Jcal_{\partial g_{\xi^g_k}}$
is interpreted as a ``backward'' step (in the context of differential equations, this would be 
an implicit Euler step).

\begin{algorithm}    
\SetKwInOut{Output}{Initialization}
  \Output{Set $X_{0} \sim \mu_0 \in \mathscr{P}_2(G)$, $X_0 \indep (\xi_k)$ 
   with $\xi_k\equiv(\xi^f_{k}, \xi^g_{k})\sim \xi\equiv(\xi^f, \xi^g)$
  i.i.d.  with values in $I_f\times I_g$, and set ${t}>0$.}
    \For{$k=0,1,2,\ldots$}{
            { 
            \begin{equation}\label{eq:spcd}
                X_{k+1}= T^{FB}_{\xi_k}X_k\equiv \Jcal_{\partial g_{\xi^g_k}}\paren{X_{k}-{t}\nabla f_{\xi^f_k}(X_{k})}
            \end{equation}
            }\\
    }
  \caption{Stochastic Forward-Backward Splitting}\label{algo:sfb}
\end{algorithm}

When $f_{\xi^f}(x) = f(x) + \xi^f\cdot x$ and $g_{\xi^g}$ is the zero function, then this is just 
 steepest descents with linear noise discussed in Section \ref{sec:consist RFI}.  More generally, 
 \eqref{eq:spcd} with $g_{\xi^g}$ the zero function models stochastic gradient 
 descents, which is a central algorithmic template in many applications.  We show how the 
 approach developed above opens the door to an analysis of this basic algorithmic paradigm for 
 {\em nonconvex} problems.  The next statement is the nonconvex analog to 
 \cite[Proposition 4.1]{HerLukStu22a}. The main difference here is that the 
 violation given by \eqref{eq:sfb violation} depends on the step size 
 $t$ in \eqref{eq:spcd}.  We will have more to say about this below.  
 Part (i) of the statement is the nonconvex analog to 
\cite[Proposition 4.1(ii)]{HerLukStu22a};  part (iv) 
 of the statement covers the convex case already established 
 in \cite[Proposition 4.1(iii)]{HerLukStu22a};  this is included for completeness.
 
\begin{prop}\label{t:sfb}
On the Euclidean space $(\Rn, \|\cdot\|)$ suppose the following hold:
\begin{enumerate}[(a)]
 \item for all $i\in I_f$, $\nabla f_i$ is Lipschitz continuous with constant $L$ on $G\subset\Rn$ and
 {\em hypomonotone} on $G$ with violation $\tau_f>0$  on $G\subset\Rn$:
\begin{equation}
 \label{e:hypomonotone} 
    -\tau_f\norm{x-y}^2\leq \ip{\nabla f_i(x)-\nabla f_i(y)}{x-y}
 \qquad \forall x, y\in G.
\end{equation}
\item there is a $\tau_g$ such that for all $i\in I_g$, the (limiting) 
subdifferential $\partial  g_i$ satisfies
\begin{equation}\label{e:submonotone}
 -\tfrac{\tau_g}{2}\norm{(x^++z)-(y^++w)}^{2}\leq \ip{z-w}{x^+-y^+}.
\end{equation}%
at all points $(x^+, z)\in \gph\partial  g_i$ and  $(y^+, w)\in \gph\partial  g_i$
where $z = x-x^+$ for $\{x^+\}= \Jcal_{\partial g_i}(x)$ for any 
$x\in \bigcup_{i\in I_f}\paren{\Id - t\nabla f_i}(G)$ and 
where $w = y-y^+$ for $\{y^+\}= \Jcal_{\partial g_i}(y)$ for any 
$y\in \bigcup_{i\in I_f}\paren{\Id - t\nabla f_i}(G)$.
\item $T^{FB}_i$ is a self-mapping on $G\subset\Rn$ for all $i$. 
\end{enumerate}
 Then  the following hold.
 \begin{enumerate}[(i)]
  \item\label{ex:spg ncvx i} $T^{FB}_i$ is a$\alpha$-fne on $G$ with constant 
  $\alpha = 2/3$ and violation at most 
  \begin{equation}\label{eq:sfb violation}
    \epsilon = \max\{0, (1+2\tau_g)\paren{1+t(2\tau_f+2tL^2)} - 1\}
  \end{equation}
  for all $i\in I$.
  \item\label{ex:spg ncvx ii} $\Phi(x, i):=T_{i}x$ is a$\alpha$-fne in expectation 
  on $G$ with 
  constant $\alpha = 2/3$ and violation at most $\epsilon$ given in \eqref{eq:sfb violation}.
  \item\label{ex:spg ncvx iv} The Markov operator $\mathcal{P}$ corresponding to 
  \eqref{eq:spcd} is a$\alpha$-fne in 
  measure on $\mathscr{P}_2(G)$ with constant $\alpha=2/3$ and violation no 
  greater than $\epsilon$ given in \eqref{eq:sfb violation}, 
  i.e. it  satisfies \eqref{eq:alphfne meas}.
  \item\label{ex:spg ncvx iii} Suppose that 
  assumption (a) holds with condition \eqref{e:hypomonotone} being 
  satisfied for  
  $\tau_f<0$ (that is, $\nabla f_i$ is strongly monotone for all $i$), and that  
  condition \eqref{e:submonotone} holds with $\tau_g=0$ (for instance, when $g_i$ is convex).  
  Then,  whenever there exists an invariant measure 
  for the Markov operator $\mathcal{P}$ corresponding to \eqref{eq:spcd}, 
  for any fixed step length $t\in (0,2\alpha/L]$
  the distributions of the sequences of random variables
  converge to an invariant measure in the Prokhorov-L\'evy metric.
  \item\label{ex:spg ncvx v} Let $G$ be compact and  $\mathscr{P}_2(G)\cap\inv\mathcal{P}\ne \emptyset$.  
  If  $\Psi$ given by \eqref{eq:Psi} 
 takes the value $0$ only at points in $\inv\mathcal{P}$ and 
 is metrically subregular for 
  $0$ on $\mathscr{P}_2(G)$ with gauge $\rho$ given by \eqref{eq:gauge} 
  with $\tau=1/2$, $\epsilon$ satisfying 
  \eqref{eq:sfb violation}, and  $\theta_{\tau,\epsilon}$
  satisfying \eqref{eq:theta} and \eqref{eq:theta summable}  
  where $t_0\equiv d_{W_2}\paren{\mu_0, \inv\mathcal{P}\cap\mathscr{P}_2(G)}<\tbar$ for all $\mu_0\in \mathscr{P}_2(G)$,
 then the Markov chain converges to an invariant distribution in the $W_2$ metric 
 with rate $O(s_k(t_0))$ 
where\\ $s_k(t_0)\equiv\lim_{N\to\infty}\sum_{j=k}^N \theta_{\tau,\epsilon}^{(j)}(t_0)$.  
 \end{enumerate} 
\end{prop}
Before proving the statement, some background for conditions \eqref{e:hypomonotone} and \eqref{e:submonotone} 
might be helpful.  The inequality \eqref{e:hypomonotone} is satisfied by functions $f$ that are 
{\em prox-regular} \cite{PolRock96a}.  This traces back to Federer's study of 
curvature measures \cite{Federer59} where such functions would be called functions whose epigraphs have 
{\em positive reach}. Inequality \eqref{e:submonotone} is equivalent to the property 
that $\Jcal_{\partial g_i}$ is a$\alpha$-fne with constant $\alpha_i=1/2$ and 
violation $\tau_g$ on $G$ \cite[Proposition 2.3]{LukNguTam18}.  Any differentiable function $g_i$ with 
gradient satisfying \eqref{e:hypomonotone} with constant $\tau_g/(2(1+\tau_g))$ 
will satisfy \eqref{e:submonotone} with constant $\tau_g$.  
In the present setting, if 
$g_i$ is {\em prox-regular} on $G$, then  
$\partial {g_i}$ is hypomonotone on $G$ and therefore 
satisfies \eqref{e:submonotone} \cite{LukNguTam18}.  Convex functions 
are trivially hypomonotone with constant $\tau=0$.  If the functions $g_i$ are convex, then 
the violation \eqref{eq:sfb violation} can be made arbitrarily small by taking the step size $t$
small enough.
\begin{proof}
\eqref{ex:spg ncvx i}.  This is \cite[Proposition 3.7]{LukNguTam18}. 

\eqref{ex:spg ncvx ii}. This follows immediately from Part \eqref{ex:spg ncvx i} above and 
Proposition \ref{r:nonneg psi_Phi}.

\eqref{ex:spg ncvx iv}.  This follows immediately from Part \eqref{ex:spg ncvx ii} above and 
Proposition \ref{thm:Tafne in exp 2 pafne of P}.

\eqref{ex:spg ncvx iii}.  This is \cite[Proposition 4.1(iii)]{HerLukStu22a}.

\eqref{ex:spg ncvx v}.  This follows from Part \eqref{ex:spg ncvx iv} and Theorem \ref{t:msr convergence}.
\end{proof}
The compactness assumption on 
$G$ in part \eqref{ex:spg ncvx v} is just to permit the application of Theorem \ref{t:msr convergence}.  
As noted in Remark \ref{r:G compact} this assumption can be dropped for mappings $T_i$ on 
Euclidean space. 
The result narrows the work of proving convergence of stochastic forward-backward 
algorithms to verifying existence of $\inv \Pcal$.  The case of convex
stochastic gradient descent was presented in \cite[Proposition 4.2]{HerLukStu22a}. 
  
\subsection{Stochastic Douglas-Rachford}  
Another prevalent algorithm for nonconvex problems is the 
Douglas-Rachford algorithm \cite{LionsMercier79}.  This is
based on compositions of {\em reflected resolvents}:
\begin{equation}\label{eq:Rprox}
 R_{f}\equiv 2\Jcal_{\partial f} -\Id.
\end{equation}

  \begin{algorithm}    
\SetKwInOut{Output}{Initialization}
  \Output{Set $X_{0} \sim \mu_0 \in \mathscr{P}_2(G)$, $X_{0} \indep (\xi_k)$ 
   with $\xi_k\equiv (\xi^f_{k}, \xi^g_{k})\sim \xi\equiv(\xi^f, \xi^g)$ i.i.d.
  with values in $I_f\times I_g$.}
    \For{$k=0,1,2,\ldots$}{
            { 
            \begin{equation}\label{eq:sdr}
                X_{k+1}= T^{DR}_{\xi_k}X_k\equiv 
                \frac{1}{2}\paren{R_{f_{\xi^f_k}}\circ R_{g_{\xi^g_k}} + \Id}(X_{k})
            \end{equation}
            }\\
    }
  \caption{Stochastic Douglas-Rachford Splitting}\label{algo:sdr}
\end{algorithm}
Algorithm \ref{algo:sdr} has been studied for solving large-scale, convex optimization 
and monotone inclusions (see for example \cite{Pesquet19} and  \cite{Cevher2018}). Proposition 
\ref{t:sdr} below 
opens the analysis to nonconvex, nonmonotone problems;
the setting for this statement is given by the following assumptions. 
\begin{assumption}\label{ass:DR}
 On the Euclidean space $(\Rn, \|\cdot\|)$ the following assumptions hold.
\begin{enumerate}[(a)]
 \item There is a $\tau_g$ such that for all $i\in I_g$, the (limiting) 
subdifferential $\partial  g_i$ satisfies
\begin{equation}\label{e:hypomonotone g}
 -\tfrac{\tau_g}{2}\norm{(x^++z)-(y^++w)}^{2}\leq \ip{z-w}{x^+-y^+}
\end{equation}%
at all points $(x^+, z)\in \gph\partial  g_i$ and  $(y^+, w)\in \gph\partial  g_i$
where $z = x-x^+$ for $\{x^+\}= \Jcal_{\partial g_i}(x)$ for any $x\in G\subset \Rn$ and 
where $w = y-y^+$ for $\{y^+\}= \Jcal_{\partial g_i}(y)$ for any $y\in G$.
 \item There is a $\tau_f$ such that for all $i\in I_f$, the (limiting) 
subdifferential $\partial  f_i$ satisfies
\begin{equation}\label{e:hypomonotone f}
 -\tfrac{\tau_f}{2}\norm{(x^++z)-(y^++w)}^{2}\leq \ip{z-w}{x^+-y^+}
\end{equation}%
at all points $(x^+, z)\in \gph\partial  f_{i}$ and  $(y^+, w)\in \gph\partial  f_{i}$
where $z = x-x^+$ for $\{x^+\}= \Jcal_{\partial f_{i}}(x)$ for any $x\in \bigcup_{j\in I_g}\{\Jcal_{\partial g_{j}}(G)\}$ and 
where $w = y-y^+$ for $\{y^+\}= \Jcal_{\partial f_{i}}(y)$ for any $y\in \bigcup_{j\in I_g}\Jcal_{\partial g_{j}}(G)$.
\item $T^{DR}_i$ is a self-mapping on $G\subset\Rn$ for all $i$. 
\end{enumerate}
\end{assumption}

\begin{prop}\label{t:sdr}
Under Assumption \ref{ass:DR} the following hold.
 \begin{enumerate}[(i)]
  \item\label{ex:sdr ncvx i} For all $i\in I_f\times I_g$ the mapping $T^{DR}_i$ 
  defined by \eqref{eq:sdr} 
  is a$\alpha$-fne on $G$ with constant 
  $\alpha = 1/2$ and violation at most 
  \begin{equation}\label{eq:sdr violation}
    \epsilon = \tfrac{1}{2}\paren{(1+ 2\tau_g)(1+2\tau_f) - 1}
  \end{equation}
  on $G$.
  \item\label{ex:sdr ncvx ii} $\Phi(x, i):=T^{DR}_{i}x$ is a$\alpha$-fne in expectation with 
  constant $\alpha = 1/2$ and violation at most $\epsilon$ given by \eqref{eq:sdr violation}.
  \item\label{ex:sdr ncvx iv} The Markov operator $\mathcal{P}$ corresponding to 
  \eqref{eq:sdr} is a$\alpha$-fne in 
  measure with constant $\alpha=1/2$ and violation no greater than $\epsilon$ given by 
  \eqref{eq:sdr violation}, 
  i.e. it  satisfies \eqref{eq:alphfne meas}.
  \item\label{ex:sdr ncvx iii} Suppose that 
  parts (a) and (b) of Assumption \ref{ass:DR} hold with conditions \eqref{e:hypomonotone g} and \eqref{e:hypomonotone f} 
  being satisfied for  
  $\tau_g=\tau_f=0$ (i.e., when $f_i$ and $g_i$ are convex for all $i$).  
  Then,  whenever there exists an invariant measure 
  for the Markov operator $\mathcal{P}$ corresponding to \eqref{eq:sdr}, 
  the distributions of the sequences of random variables
  converge to an invariant measure in the Prokhorov-L\'evy metric.  
  \item\label{ex:sdr ncvx v} Let $G$ be compact and  $\mathscr{P}_2(G)\cap\inv\mathcal{P}\ne \emptyset$.
  If $\Psi$ given by \eqref{eq:Psi} 
 takes the value $0$ only at points in $\inv\mathcal{P}$ and 
 is metrically subregular for 
  $0$ on $\mathscr{P}_2(G)$ with gauge $\rho$ given by \eqref{eq:gauge} 
  with $\tau=1/2$, $\epsilon$ satisfying 
  \eqref{eq:sdr violation}, and  $\theta_{\tau,\epsilon}$
  satisfying \eqref{eq:theta}  and \eqref{eq:theta summable}  
  where $t_0\equiv d_{W_2}\paren{\mu_0, \inv\mathcal{P}\cap\mathscr{P}_2(G)}<\tbar$ for all $\mu_0\in \mathscr{P}_2(G)$,
 then the Markov chain converges in the $W_2$ metric 
 with rate $O(s_k(t_0))$ where $s_k(t_0)\equiv\lim_{N\to\infty}\sum_{j=k}^N \theta_{\tau,\epsilon}^{(j)}(t_0)$.  
 \end{enumerate} 
\end{prop}
\begin{proof}
\eqref{ex:sdr ncvx i}.  By \cite[Proposition 3.7]{LukNguTam18} for all $j\in I_g$, 
$\Jcal_{\partial g_j}$ is a$\alpha$-fne 
with constant $\alpha=1/2$ and violation $\epsilon_g = 2\tau_g$ on $G$.  
Likewise, for all $i\in I_f$,  $\Jcal_{\partial f_i}$  
is a$\alpha$-fne 
with constant $\alpha=1/2$ and violation $\epsilon_f = 2\tau_f$ on 
$\bigcup_{j\in I_g}\{\Jcal_{\partial g_{j}}(G)\}$.
By \cite[Propositions 2.3-2.4]{LukNguTam18}, for all 
$i\in I_f\times I_g$ the Douglas-Rachford mapping $T^{DR}_i$ is therefore 
a$\alpha$-fne with constant $\alpha=1/2$ and 
violation at most $\tfrac{1}{2}\paren{(1+2\tau_g)(1+2\tau_f) - 1}$ 
on $G$.  

\eqref{ex:sdr ncvx ii} - \eqref{ex:sdr ncvx v} follow in the same way as their 
conterparts in Proposition \ref{t:sfb}.
\end{proof}
Here as in Proposition \ref{t:sfb} the compactness assumption on $G$ 
in part \eqref{ex:sdr ncvx v} can be dropped
since $T_i$ is a mapping on $\Rn$ (see Remark \ref{r:G compact}). 

\subsection{Application to X-FEL Imaging}
We apply the Stochastic Forward-Backward Algorithm \ref{algo:sfb} 
and the Stochastic Douglas-Rachford Algorithm \ref{algo:sdr} to 
the problem of X-ray free electron laser imaging discussed in the introduction.  

Here, a high-energy X-ray pulse illuminates molecules suspended in fluid that is 
streaming across the beam.  A 
low-count diffraction image is recorded for each pulse.  
The goal is to reconstruct the three-dimensional electron density of the target molecules 
from the observed diffraction images.  This is a stochastic 
tomography problem with a nonlinear model for the data - stochastic because the molecule 
orientations are random, and uniformly distributed on SO(3).  
Computed tomography with random orientations has been studied for more than two decades 
(see \cite{Bresler2000a} and  \cite{Bresler2000b}) and 
been successfully applied for inverting the Radon transform (a linear operator) with unknown orientations 
\cite{Panaretos09} and  \cite{SingerWu13}.  
The model for the data in X-FEL imaging is nonlinear and nonconvex: Fraunhoffer diffraction 
(mathematically equivalent to Fourier transformation) with missing phase 
\cite{Born} and \cite{ArdGru20}.  The problem of recovering an object from diffraction 
intensity data is the optical phase retrieval problem.  The most successful 
and widely applied methods for solving this problem are fixed point algorithms where the 
fixed point mappings consist of compositions and averages of projection mappings onto nonconvex sets 
\cite{LukSabTeb19};  the connection to the general framework considered here is through the fact that 
the projector is the proximal mapping of the indicator function \cite{VA}, and 
the gradient of the squared distance of a point to a set is twice the difference between that point 
and its projection onto the set \cite[Eq(5.3)]{Luke02a}.  With additional constraints ensuring that 
the reconstructions are confined to a certain region, or are real-valued, 
the feasibility model for the X-FEL problem, and phase retrieval in general, is an example of an 
{\em inconsistent}, nonconvex feasibilty problem:  there
does not exist a point that simultaneously explains the measurement and satisfies the a priori constraints. 
A theoretical framework for unifying and extending the first proofs of local convergence 
of projection methods for inconsistent, nonconvex feasibility, with rates, was established in \cite{LukNguTam18}.  
This analysis accommodates iterations of averages and compositions of expansive mappings.  Moreover, 
unlike many other approaches, 
the framework does not require that the constituent mappings have common fixed points.  
This has been applied to prove, for the first time, local linear convergence of 
a wide variety of fundamental algorithms for phase retrieval (see \cite{LukMar20}, 
 \cite{Thao18},  \cite{HesseLuke13}, and \cite{LukNguTam18}). 

Concretely, the problem is to determine the three-dimensional electron density $\rho$ from a collection of two-dimensional 
images $\{Y_{s_1}, Y_{s_2}, \dots, Y_{s_m}\}$, where each $Y_{s_j}$ is a sample from a 
two-dimensional cross section of the three-dimensional probability distribution $\phi$ at 
orientation $s_j\in SO(3)$, denoted $\phi^{s_j}$\footnote{In this presentation, we ignore 
the additional complication that the rotation $s$ is unobservable and must be 
estimated, one way or another, conditioned on the current estimate for the density.}.  
The set of possible densities $\rho$ satisfying such measurements 
is given by
	\begin{equation}\label{eq:C}
		C\paren{\{Y_{s_1}, Y_{s_2}, \dots, Y_{s_m}\}} \equiv 
		\set{\rho}{\mathbb{P}\paren{\{Y_{s_1}, Y_{s_2}, \dots, Y_{s_m}\}~|~
		\phi}\approx 1}
	\end{equation}
where  
$\mathbb{P}\paren{\{Y_{s_1}, Y_{s_2}, \dots, Y_{s_m}\}~|~{\phi}}$ 
is the probability of making the collection of observations 
$\{Y_{s_1}, Y_{s_2}, \dots, Y_{s_m}\}$ conditioned on $\phi$.  The discretized 
mathematical model for $\phi$ at a voxel $i$ in a given cross section 
$\phi^{s}$ is
	\begin{equation} \label{eq:physical model}
		\left|\paren{\Fcal^s(\rho)}_{i}\right|^2 = \phi^s_{i}.
	\end{equation}
Here $\mymap{\Fcal^s}{\mathbb{C}^{n}}{\mathbb{C}^{n_s}}$ is the cross sectional 
Fourier transform accounting for the propagation of an electromagnetic wave with 
incident angle $s$ onto a two-dimensional surface passing through $n_s$ of the voxels. 
Physically, $\phi^s_{i}$ is the {\em intensity} of the field
at voxel $i$, and accounts for the probability of observing a scattered photon at this voxel 
in the far field under the rotation $s$. 

The set $C\paren{\{Y_{s_1}, Y_{s_2}, \dots, Y_{s_m}\}}$ has been studied extensively in 
\cite{Luke02a},  \cite{BurkeLuke03},  \cite{Luke12},  
\cite{LukNguTam18}, and  \cite{LukMar20} in the case where the observations $\{Y_{s_1}, Y_{s_2}, \dots, Y_{s_m}\}$
together can be taken to be the observed intensity mapping $\phi$, that is 
$C\paren{\{Y_{s_1}, Y_{s_2}, \dots, Y_{s_m}\}}=C(\phi)$. 
Although this set is nonconvex, it is prox-regular \cite{Luke12}. 
Projectors onto $C(\phi)$ are therefore pointwise almost $\alpha$-firmly nonexpansive at 
	any point $\overline{\rho}\in C(\phi)$ with violation $\epsilon$ vanishing to zero as the 
	neighborhood of $\overline{\rho}$ collapses. 

To give an idea of the size of this problem, in a typical experiment $n$ is about $10^8$.  	
In practice, the observation described above for a randomly selected orientation $s_k$ is repeated
about $10^9$ times, each time with a different orientation.  At present it is not feasible to process 
all $10^9$ images $\{Y_{s_1}, Y_{s_2}, \dots, Y_{s_M}\}$ in the construction of the set \eqref{eq:C}.  
Instead, we propose randomly subsampling 
the complete set of images.  

In addition to the sets generated by the data, there are certain a priori qualitative constraints that can 
(and should) be added depending on the type of experiment that has been conducted. Often these are 
support constraints, or real-valuedness, or nonnegativity;  an electron density, for example is by definition 
real-valued and nonnegative.  All of these are convex constraints for which
we reserve the set $C_{0}$ for the qualitative constraints.
	
The problem is a specialization of \eqref{eq:opt prob} where $I_f=\{1\}$, 
    $i^g\in I_g$ indexes all subsets $\mathbb{Y}_{i^g}\subset \{Y_{s_1}, Y_{s_2}, \dots, Y_{s_M}\}$
    with a fixed cardinality $|\mathbb{Y}_{i^g}| = m<M$, 
    $\xi^f_k=1$ for all $k$, $\xi^g_k$ is a 
	uniformly distributed random variable on $I_g$ for all $k$, and 
	\begin{eqnarray} 
	 (\forall k)\quad f_{\xi^f_k}(\rho)&\equiv& \tfrac{\lambda}{2(1-\lambda)}\dist^2(\rho, C_0)\nonumber\\
	 g_{\xi^g_k}(\rho)&\equiv& \iota_{C(\{Y_{s_j}\}_{j\in \mathbb{Y}_{\xi^g_k}})}(\rho)\equiv 
	 \begin{cases} 0&\mbox{ if } \rho\in C(\{Y_{s_j}\}_{j\in \mathbb{Y}_{\xi^g_k}})\\
	                                                     +\infty& \mbox{ otherwise}.
	                                                    \end{cases}
	\end{eqnarray}

	Both Algorithm \ref{algo:sfb} and Algorithm \ref{algo:sdr} can be applied to this problem.  	
	Assumptions (a) and (b) of both Proposition \ref{t:sfb} and Proposition \ref{t:sdr} hold  
	on all neighborhoods of the sets $C_0$ and $C(\{Y_{s_j}\}_{j\in \mathbb{Y}_{\xi^g_k}})$ small enough 
	\cite[Example 3.6]{LukNguTam18}.  
	Indeed, for this application $\tau_f=0$ since $C_0$ is convex and $\tau_{g_{\xi^g_k}}$ in 
	\eqref{e:hypomonotone g} can be estimated from the distance of $\rho$ to the set 
	$C(\{Y_{s_j}\}_{j\in \mathbb{Y}_{\xi^g_k}})$, which 
	is reasonably easy to calculate.  Deterministic versions of Algorithm \ref{algo:sfb} have been fully studied in 
	\cite{LukNguTam18}, though we suspect that the occurrence of local minima could be a significant problem for 
	this algorithm applied to this problem.  In \cite{LukMar20} the fixed points of the 
	deterministic version of Algorithm \ref{algo:sdr} for the  
	phase retrieval problem have been characterized, and metric subregularity of the
	transport discrepancy \eqref{eq:delta} has been determined for geometries applicable 
	to {\em cone and sphere} problems \cite{LukSabTeb19} such as this.  So 
	for a majority of relevant instances, there is good reason to expect that 
	Propositions \ref{t:sfb} and \ref{t:sdr} can be  
	applied provably to X-FEL measurements.  The determination of the local domain $G$ in 
	condition (c) 
	of Proposition \ref{t:sfb} and Assumption \ref{ass:DR} 
	is therefore key.   
	There are some unresolved cases, however, that 
	are relevant for optical phase retrieval (see \cite[Example 5.4]{LukMar20}), 
	and this needs further study.  

\subsection{Stochastic Proximal Algorithms in Tree Space}\label{sec:tree}
We conclude with proximal splitting in locally compact Hadamard spaces. Our target application in this setting
is the computation of Fr\'echet means in what is known as the BHV space, a model space for phylogenetic trees \cite{BilHolVog}.
This is a Hadamard space consisting of stratified Euclidean spaces.  Each strata of the space consists of all trees 
(connected, acyclic,  undirected graphs) with 
the same sets of nodes and edges, but different edge lengths connecting the various nodes.  The leaves of such trees are 
individual species, and the nodes connecting the species are points of common genetic ancestry;  the lengths of the 
edges between nodes represents the time required for the observed genetic variation to occur.  A particular 
tree relationship depends on what part of the genome is being compared;  measurements from two different parts of 
the genome will generally lead to two different phylogenetic trees.  
If two trees belong 
to the same strata, i.e. if the trees differ only in edge lengths, then a geodesic connecting these trees is simply a 
line between two points in a Euclidean space;  otherwise, 
the geodesic passes from one stratum to another by passing through the origin, i.e. by shrinking an edge length to zero
and/or by extending an edge from the origin.  

The {\em distance} between two points $x$ and $y$ in this space, $d(x,y)$ 
is (obviously) the length of the geodesic connecting them, and this can be decomposed into the sum 
of piecewise Euclidean distances of the geodesic confined to individual strata.  The distance function is 
clearly convex.  The task we consider in this example is that of computing the average, 
or more precisely the {\em Fr\'echet mean},
of more than two trees in 
the BHV space:  Given $N$ trees, $x_1, x_2, \dots, x_N$ find $x^*$ that solves
\begin{equation}
 \label{eq:Frechet mean} \inf_{x}\sum_{i=1}^N d(x,x_i)^2.
 \end{equation}
Readers interested in learning more about the application to 
phylogenetic trees are referred to \cite{BilHolVog} and references therein;  our primary interest is in the mathematical 
structure of the problem and demonstrating the theory developed in the previous sections. 

Let $(H,d)$ be a Hadamard space, and let $f_i:H\to \mathbb{R}$ be proper, lower semicontinuous
convex functions for $i=1,2,\dots N$.  The Fr\'echet mean problem is an instance of the following 
convex optimization problem
\begin{equation}
 \label{eq:sum of cvx} \inf_{x\in H}\sum_{i=1}^N f_i(x).
 \end{equation}
 In a Hadamard space the proximal mapping of a function 
 $f$ is defined by \eqref{eq:prox} with the only difference that the domain of $f$ is $H$. 
 Note that the prox mapping is well-defined, even though the resolvent of the subdifferential 
 is not since $H$ is not a linear space.  
 This has been studied in $CAT(0)$ spaces in \cite{Jost97},  \cite{Banert} and  \cite{AriLeuLop14}, 
 and in the Hilbert ball in \cite{KopRei09}.   
 In these earlier works it was already known that 
proximal mappings  of lower semicontinuous convex functions are (everywhere) $\alpha$-firmly nonexpansive with 
$\alpha=1/2$.  In \cite{Bacak14} a randomized proximal splitting algorithm is studied
under the assumption that the proximal parameters $\lambda$ in \eqref{eq:prox} decay to zero, 
in which case the prox mapping converges to the identity mapping.  The theory developed above
allows us to conclude convergence to an invariant probability measure without the assumption 
that the prox mapping converges to the identity.  

Letting $\xi$ be a uniformly distributed random variable with 
values on $\{1, 2, \dots, N\}$ the {\em backward-backward splitting} method applied to this problem yields 
 Algorithm \ref{alg:bbs}.
\begin{algorithm}    
\SetKwInOut{Output}{Initialization}
  \Output{Given $f_1\ldots,f_m$ and $\lambda_i >0$ $(i=1,2,\dots,m)$, 
  set $X_{0} \sim \mu_0 \in \mathscr{P}(G)$, $X_{0} \indep (\xi_{k})$ 
   with  $\xi_k\sim\xi$ i.i.d.} 
    \For{$k=0,1,2,\ldots$}{
            {    \begin{align*}
		X_{k+1}=\prox_{f_{\xi_k}, \lambda_{\xi_k}}(X_k)
    \end{align*}
}\\
    }
  \caption{Proximal splitting}\label{alg:bbs}
\end{algorithm}
In the deterministic 
setting where the prox mappings are selected in a fixed sequential manner, convergence 
has been established already in \cite[Theorem 4.1]{RuiLopNic15}. Local linear convergence
was established in \cite[Theorem 27]{BLL} under the assumption of linear metric subregularity, 
which in this setting reduces to 
 \begin{equation}\label{eq:simple msr}
  d(x,\Fix \overline{T}_m \cap G)\leq r d(\overline{T}_m x, x))\quad \forall x\in G\subset H 
\end{equation}
where 
\[
\overline{T}_m\equiv  
		\paren{\prox_{f_m, \lambda_m}\circ\cdots\circ \prox_{f_2, \lambda_2}\circ \prox_{f_1, \lambda_1}}. 
\]
A random version
of this algorithm with diminishing constants $\lambda_{\xi_k}$ was shown to converge in 
\cite[Theorem 3.7]{Bacak14}.  

The immediate contribution of our work to these earlier studies is algorithmic: by Corollary 
\ref{t:msr convergence - linear} our random implementation 
{\em without diminishing constants} converges linearly to a stationary distribution under the assumption
$\Psi^{-1}(0)=\inv\mathcal{P}$
 and $\Psi(\pi)=0$ for all $\pi\in\inv\mathcal{P}$, and for all
$\mu\in\mathscr{P}_{2}(G)$ 
\begin{eqnarray}
\inf_{\pi\in\inv\mathcal{P}} W_2(\mu, \pi) 
 &\le& rd_{\mathbb{R}}(0,\Psi(\mu))
 = r\Psi(\mu).\nonumber
\end{eqnarray}
In a Euclidean setting this assumption for this problem can be shown to hold
by demonstrating that the discrete set of points $\{x_1, \dots,x_N\}$  is (trivially) 
{\em subtransversal} \cite[Definition 3.2 and Proposition 3.4]{LukNguTam18}.  We conjecture
that this holds in the present setting of randomized algorithms on Hadamard spaces, but 
to show this would take more space than we have here.

This example does not use the full potential of our framework since this problem is convex, and hence 
the proximal mappings are all $\alpha$-firmly nonexpansive with constant $\alpha=1/2$.  
One prominent nonconvex problem in this setting to which our framework can be applied is the 
$K$-means clustering problem:  find the best assignment of a sample of $N$ trees 
$\{x_1, x_2, \dots,x_N\}$ into $K$ clusters.  In this problem one begins with $N$ clusters (the individual trees) and 
reduces the number of clusters by assigning nearest neighbors to a single cluster whose 
center is the Fr\'echet mean.  The 
operation of assigning a tree to a cluster is expansive since two trees arbitrarily close to a point that is
equidistant to more than one cluster will be assigned to distinct clusters with possibly very distant 
Fr\'echet means.  


\end{document}